\documentclass[12pts]{article}
\usepackage[utf8]{inputenc}
\usepackage[T1]{fontenc}
\usepackage{macros}
\usepackage[english]{babel}
\usepackage{amsmath}
\usepackage{tikz-cd}
\usepackage{tikz}
\usepackage{amsthm}
\usepackage{amssymb}
\usepackage{stmaryrd}
\usepackage{amsfonts}
\usepackage{graphicx}

\usepackage{caption}
\usepackage{subcaption}

\usepackage{xspace}
\usepackage[all]{xy}
\usepackage{txfonts}
\usepackage{csquotes}
\usepackage{mathtools}
\usepackage{thmtools}
\usepackage{tdsfrmath}
\usepackage{xfrac}
\usepackage[shortlabels]{enumitem}
\usepackage{geometry}
\usepackage{arcs}
\usepackage{txfonts}

\newtheorem{thm}{Theorem}[section]
\newtheorem{lem}[thm]{Lemma}
\newtheorem{cor}[thm]{Corollary}
\newtheorem{mainthm}{Main Theorem}

\theoremstyle{remark}
\newtheorem{rem}{Remark} [section]
\newtheorem{note}[rem]{Notation}

\newtheorem{prop}[thm]{Proposition}

\newtheorem*{claim}{Claim}

\fontfamily{lmss}

\theoremstyle{remark}
\newtheorem{step}{Step}

\theoremstyle{definition}
\newtheorem{defi}[thm]{Definition}
\newtheorem{Ex}[thm]{Example}

\newtheorem*{thm*}{Theorem}
\newtheorem*{cor*}{Corollary}
\newtheorem*{conj*}{Conjecture}

\DeclareFontFamily{OMX}{yhex}{}
\DeclareFontShape{OMX}{yhex}{m}{n}{<->yhcmex10}{}
\DeclareSymbolFont{yhlargesymbols}{OMX}{yhex}{m}{n}
\DeclareMathAccent{\wip}{\mathord}{yhlargesymbols}{"F3}

\title{Strong collapsibility of the arc complexes of orientable and non-orientable crowns}
\author{Pallavi Panda}
\date{}

\begin{document}
	\maketitle
	\paragraph{Abstract. } We prove that the arc complex of a polygon with a marked point in its interior is a strongly collapsible combinatorial ball. We also show that the arc complex of a M\"{o}bius strip, with finitely many marked points on its boundary, is a simplicially collapsible combinatorial ball but is not strongly collapsible. 
	\section{Introduction}
	%simple collapse
	In this paper, we study the topology and the collapsibility of three families of finite simplicial complexes called the arc complexes. 
	
\paragraph{Collapsibility.}In 1939, Whitehead \cite{whitehead} introduced the notion of \emph{simplicial collapse} of a finite simplicial complex. It is the operation of removing the open star of a simplex that is contained in a unique maximal simplex  (see Section \ref{simtop} for the definition). The remaining complex is then simple homotopy equivalent to the original complex.  A simplicial complex is said to be \emph{collapsible} if there is a finite sequence of simplicial collapses resulting in a 0-simplex. Collapsibility is the combinatorial version of contractibility of topological spaces. It was studied extensively by Bing, Cohen, Lickorish and Zeeman. This resulted in examples of contractible complexes that were not collapsible: Zeeman's Dunce hat \cite{zeeman}, Bing's house \cite{bing}. Welker \cite{welker} proved that if a complex is collapsible then its barycentric subdivision is \emph{non-evasive}. Every non-evasive complex is collapsible but the converse is false. He also showed that the join of two complexes is collapsible if at least one of them is collapsible. 
\paragraph{Collapsible manifolds.}An important direction of research has been to find a relationship between topological balls and the condition of collapsibility of manifolds. Whitehead proved that all collapsible piecewise-linear (PL) $d$-manifolds are $d$-balls. All 2-balls are collapsible but for $d\geq 3$, there exist PL $d$-balls that are not collapsible. In \cite{extremal}, authors Adiprasito--Benedetti--Lutz construct an example of a 5-manifold that is collapsible but not homeomorphic to the 5-ball. In \cite{dmtbdry}, Benedetti shows that all \emph{endo-collapsible} manifolds with boundary are balls using the concept of duality in discrete Morse theory. Crowley \cite{crowley} proved that all 3-dimensional pseudomanifolds that are CAT(0) with some equilateral metric, are collapsible. Benedetti--Adiprasito \cite{catzero} generalised this result to higher dimensions: any complex that is CAT(0) with a metric for which all vertex stars are convex, is collapsible. 
%strong collapseA
\paragraph{Strong collapses.}A strengthening of the notion of a simplicial collapse is given by a strong collapse. 
In \cite{LC}, Matou\v sek introduced the notion of \emph{LC-removable} vertices or \emph{dominated vertices} of a simplicial complex. The link of such a vertex is a cone. The process of deleting such a vertex from the complex is called a \emph{LC-reduction} or \emph{strong collapse}. Strong collapsibility implies collapsibility but the converse is not true. Matou\v sek showed the existence and uniqueness of \emph{cores} (subcomplexes with no dominated vertices) of a finite complex. Such a core does not exist in the case of simplicial collapsibility: in \cite{benball}, Benedetti--Lutz constructed a simplicial 3-ball that is collapsible but also collapses onto a Dunce hat. In \cite{worst}, Lofano--Newman give further examples of collapsing sequences of a simplex that get stuck at a non-collapsible subcomplex. 
In \cite{strong}, Barmak--Minian study the notion of strong homotopy theory. Two complexes are said to be strongly homotopic to one another if one can be strongly collapsed onto the other. They gave a sufficient condition for strong collapsibility using the \emph{nerve} operation. They introduced the notion of strong $d$-collapse: this is the removal of a $d$-simplex whose link is a cone. This forms an intermediary between simplicial collapse and strong collapse. They also show that a simplicial complex is strong collapsible if and only if its barycentric subdivision is strongly collapsible and that the join of two simplicial complexes is strong collapsible if and only if either of them is strongly collapsible. In \cite{PHstrcoll}, Boissonnat--Pritam--Pareek give an application of strong collapses in creating an efficient algorithm to compute the persistent homology of a chain of simplicial complexes. In \cite{PHflag}, Boissonnat--Pritam show that the efficacy of their algorithm increases when restricted to flag complexes.
%arc complex

 \paragraph{Arc complexes.}A family of examples of flag PL-manifolds arises from the arc complexes of surfaces with boundary. For a finite-type orientable surface with finitely many marked points (called \emph{vertices}) on its boundary as well as its interior, one constructs its arc complex using non-trivial embedded arcs whose endpoints are among these vertices (see Section \ref{acc} for the formal definition). These were first studied by Harvey in \cite{Harvey}. The maximal simplices of this complex correspond to triangulations of the surface. Barring certain surfaces with low topological complexity, the arc complexes of most of the surfaces are locally non-compact. The exceptions are given by a convex polygon, a crown, a three-holed 2-sphere and the Möbius strip with marked points on the boundary, called \emph{non-orientable crown} in this paper. It is a classical result of combinatorics that the arc complex of a convex $n$-gon is a piecewise linear sphere of dimension $n-4$.   Penner \cite{penner} proved this result in the context of hyperbolic surfaces and Teichmüller theory. He studied the arc complex of an \emph{ideal polygon}, which is the convex hull in the hyperbolic plane $\HP$ of finitely many points (called \emph{ideal}) on the boundary $\HPb$. The diagonals in this case are bi-infinite hyperbolic geodesics with endpoints in this finite set. He attributes the original proof to Whitney. In \cite{penner}, Penner studied the topology of the quotient of the arc complex under the canonical action of the pure mapping class group of the surface. He conjectured that 
this quotient space is a sphere of a certain dimension, but this was later proved to be false by Sullivan. There is a complete list of surfaces (see \cite{sphereconj}) for which the statement is true.
The dual graph, with respect to the codimension zero and one faces, to the arc complex, called the \emph{flip graph}, has also been extensively studied. The flip graph of a convex polygon forms the 1-skeleton of the famous polytope called the \emph{associahedron}. Sleator–Tarjan–Thurston \cite{tarjan} showed that the $d$-dimensional associahedron has diameter at most $2d-4$ when $d \geq 9$ and that this bound is asymptotically exact using hyperbolic geometry. Later, Pournin \cite{diameterpournin} proved the equality for all $d\geq 9$ using combinatorial methods. In \cite{pointi}, Parlier--Pournin proved that the flip graph for a crown with $n\geq 1$ vertices is connected and its diameter is given by $2n-2$. This graph is a topological analogue of the flip-graph of a Euclidean convex polygon with an interior point.

In the case of a general topological surface with marked points, Harer showed that a specific open dense subset of the arc complex, called the \emph{pruned arc complex}, is an open ball of dimension one less than that of the deformation space of the surface.

Wilson \cite{wilson} proved that the arc complexes of a convex polygon, a punctured crown $\poly n\setminus \{0\}$ and a non-orientable crown are shellable. This is an ordering of all the maximal simplices of the complex so that the intersection of the $k$-th simplex with the union of the first $(k-1)$ simplices is always a pure complex of codimension one. He also reproved the sphericity of the arc complexes of the first two surfaces by using a result by Danaraj and Klee \cite{danaraj} which states that a shellable $d$-pseudo-manifold without boundary is a PL-sphere. Collapsible complexes are shellable. 

\paragraph{Contributions.} In this paper, firstly we show that the arc complex of a crown is strongly collapsible. 
\begin{mainthm}	For $n\geq 1$, the full arc complex $\ac{\holed n}$ of a crown $\holed n$ is strongly collapsible.
\end{mainthm}
As a corollary, we get that 
\begin{cor*}
For $n\geq 1$, the full arc complex $\ac{\holed n}$ of a crown $\holed n$ is a combinatorial ball of dimension $n-1$.
\end{cor*}

Secondly, we show that the \emph{inner arc complex} of a non-orientable crown, generated by arcs that intersect core curve of the surface, is strongly collapsible. 
\begin{mainthm}
	For $n\geq 1$, the inner arc complex $\ack C{\mob}$ of a non-orientable crown $\mob$ is strongly collapsible.
\end{mainthm}
Thirdly, we show that the full arc complex non-orientable crown is collapsible.
\begin{thm*}
For all $n\geq 1$, the simplicial complex $\ac\mob$ simplicially collapses onto $\across$.
\end{thm*}
Again, as a corollary we get that 
\begin{cor*}
For all $n\geq 1$, the simplicial complex $\ac\mob$ is a combinatorial ball of dimension $n-1$.
\end{cor*}
Finally, we prove that the full arc complex is not strongly collapsible.
\begin{mainthm}
	For $n\geq 4$, the full arc complex $\ac{\mob}$ of the surface $\mob$ is not strongly collapsible. 
\end{mainthm}
In \cite{ppballs}, we endowed convex polygons and once-punctured crowns with a bicolouring (red-blue) of the vertices and proved that the arc complex generated by the blue-blue and red-blue arcs is a shellable closed ball. In this paper we will show that in the case of a particular bicolouring of an $n$-gon (called \emph{integral strip}, see Section \ref{vocab}), the arc complex generated by red-blue arcs only is strongly collapsible.
\begin{mainthm}
	For $m,n\geq 1$ and $m+n\geq 5$, the arc complex $\acs{m,n}$ of an integral strip $P(m,n)$ is strongly collapsible.
\end{mainthm}

	\paragraph{Plan of the paper} In Section \ref{simtop}, we recall the necessary concepts on simplicial topology, collapses and strong collapses along with some results that will be used later in the proofs. In Section \ref{vocab}, we introduce the two surfaces, their arcs and their arc complexes. In Section \ref{proofs}, we give the proofs of the four theorems stated above. 
	\paragraph{Acknowledgments} This work was done in and was funded by Universit\'e Sorbonne Paris Nord. I would like to thank Lionel Pournin for his helpful comments and encouragement.

	\section{Simplicial Topology}\label{simtop}
		In this section, we will recall all the vocabulary and important results related to simplicial collapses that will be used in the proofs of this paper.
		
	\paragraph{Combinatorial manifolds. }A simplicial complex is called \emph{pure} if all of its maximal simplices have the same dimension. The \emph{dual graph} of a pure simplicial complex is the graph whose vertices are the maximal simplices and two vertices are joined by an edge if the corresponding maximal simplices share a codimension one face. A pure simplicial complex is said to be \emph{strongly connected} if its dual graph is connected. A \emph{$d$-pseudo-manifold with boundary} is a pure strongly connected $d$-simplicial complex in which every $(d-1)$-simplex is contained in atmost two $d$-simplices.
	A simplicial complex $X$ is called a \emph{combinatorial $d$-manifold with boundary} if the link of each 0-simplex is either a combinatorial $(d-1)$-sphere or a combinatorial $d-1$-ball and there exists a 0-simplex such that its link is of the latter kind. 
	%A \emph{$d$-pseudo-manifold with boundary} is a pure strongly connected $d$-simplicial complex in which every $(d-1)$-simplex is contained in atmost two $d$-simplices. Note that the boundary of such a simplicial complex is formed by all $(d-1)$-simplices that are contained in 
%\paragraph{Shelling.}Let $X$ be a pure finite simplicial complex of dimension $d$. A \emph{shelling} of $X$ is an enumeration of its maximal simplices $\mathcal{T}:( C_1,\ldots,C_n)$ such that for every $1\leq k\leq n$, the intersection $\pa{\bigcup\limits_{j=1}^{k-1}C_j} \bigcap C_k$ is a pure simplicial complex of dimension $d-1$.
	%A \emph{$d$-pseudo-manifold with boundary} is a pure strongly connected $d$-simplicial complex in which every $(d-1)$-simplex is contained in atmost two $d$-simplices. 

%	Note that the boundary of a $d$-pseudo-manifold with boundary is formed by all $(d-1)$-simplices that are contained in exactly one $d$-simplex.
%	The following is a lemma linking shellability and join of two simplicial complexes that we shall use in the proof of our main theorems. See \cite{joinshelling} for a proof.
%	%\begin{lem}\label{join}
%	%	Two complexes $X,Y$ are shellable if and only if $X\Join Y$ is shellable.
%	%\end{lem}
	%\PP {Define pure, flag, shellability, combinatorial d-manifold}
%\vspace{0.3cm}

	\paragraph{Simplicial collapses.}Let $X$ be a simplicial complex and $\sigma$ a simplex of $X$. The \emph{face-deletion} of $\sigma$ in $X$ is the subcomplex defined as $\fdel \sigma X:=\{\eta\in X \mid \sigma\nsubseteq \eta\}$. In particular, the face deletion of a 0-simplex $v$ in $X$ is denoted by $X\smallsetminus \{v\}.$		
	Let $\sigma\subsetneq\tau \in X$ be two simplices  such that $\tau$ is the only maximal simplex containing $\sigma$.  Then $\sigma$ is called a \emph{free face} of $X$. The tuple $(\sigma, \tau )$ is called a \emph{collapsible pair}. The complex $X$ is said to \emph{simplicially collapse} onto its  subcomplex $\fdel \sigma X$, if $\sigma$ is a free simplex of $X$. The complex $X$ is said to be \emph{collapsible} if there is a finite sequence of simplicial collapses leading to a 0-simplex. By $X\searrow Y$, we mean that the complex $X$ simplicially collapses onto the complex $Y$. Whitehead \ref{whitehead} showed that if $X, Y$ are two simplicial complexes such that $X\searrow Y$, then $X$ has the same homotopy type as $Y$.
%	\begin{defi}
%Let $X$ be a simplicial complex and $\sigma$ a simplex of $X$. The \emph{face-deletion} of $\sigma$ in $X$ is the simplicial complex $\fdel \sigma X:=\{\eta\in X \mid \sigma\nsubseteq \eta\}$. In particular, the face deletion of a 0-simplex $v$ in $X$ is denoted by $X\smallsetminus \{v\}.$
%	\end{defi}
%%	\begin{rem}
%%The face-deletion and the deletion of a 0-simplex coincide.
%%	\end{rem}
%
%\begin{defi}
%	Let $\sigma\subsetneq\tau$ be two simplices of $X$ such that $\tau$ is the only maximal simplex containing $\sigma$.  Then $\sigma$ is called a \emph{free face} of $X$. The tuple $(\sigma, \tau )$ is called a \emph{collapsible pair}.
%\end{defi}
%\begin{defi}
%	A complex $X$ is said to \emph{simplicially collapse} onto the complex $\fdel \sigma X$, if $\sigma$ is a free simplex of $X$. A simplicial complex is said to be \emph{collapsible} if there is a finite sequence of simplicial collapses leading to a 0-simplex.
%	\end{defi}
%By $X\searrow Y$, we denote that a complex $X$ simplicially collapses onto another complex $Y$. 
%
%\begin{thm}
%If $X, Y$ are two simplicial complexes such that $X\searrow Y$, then $X$ has the same homotopy type as $Y$.
%\end{thm}
%\PP{Prop Welker new line}
The following theorems about collapsible complexes were proved by Welker in \cite{welker}.
\begin{prop}\label{Welker}
	\begin{enumerate}[label=\alph*)]
	\item \label{join} Let $X$ and $Y$ be two simplicial complexes such that $X$ is collapsible. Then the join $X\Join Y$ is collapsible.
	\item \label{link} Let $X$ be a simplicial complex and $\sigma\in X$ be a simplex such that $\Link \sigma X$ is collapsible. Then,  $X\searrow\fdel \sigma X$.
	\end{enumerate}
\end{prop}
As mentioned  in the introduction, the following theorem was proved by Whitehead \cite{whitehead}.
\begin{thm}\label{whitehead}
A collapsible combinatorial $d$-manifold is a combinatorial $d$-ball.
\end{thm}

%%%%%%%%%%%%%%%%%%%%combinatorial ball%%%%%%%%%%%%%%%%%%%%%%%
\paragraph{Strong collapses.}A 0-simplex $v\in X$ is said to be \emph{vertex-dominated} if there exists another 0-simplex $v'\in X$ such that $\Link{v}{X}=v' \Join L,$ where $L$ is a subcomplex of $X$. In other words, any maximal simplex containing $v$ must also contain $v'$. A complex $X$ is said to \emph{strongly collapse} on $ X\smallsetminus \{v\}$ if the 0-simplex $v$ is vertex-dominated in $X$. We will denote this operation as $X\Searrow  X\smallsetminus \{v\}$.

\begin{Ex}
For $n\geq 0$, any $n$-simplex is strongly collapsible.
\end{Ex}

%\begin{defi}
%	A 0-simplex $v\in X$ is said to be \emph{vertex-dominated} if there exists another 0-simplex $v'\in X$ such that $\Link{v}{X}=v' \Join L,$ where $L$ is a subcomplex of $X$.   
%\end{defi}
%In other words, any maximal simplex containing $v$ must also contain $v'$.
%
%\begin{defi}
%A complex $X$ is said to \emph{strongly collapse} on $ X\smallsetminus \{v\}$ if the 0-simplex $v$ is vertex-dominated in $X$. We will denote this operation as $X\Searrow  X\smallsetminus \{v\}$.
%\end{defi}
%\begin{defi}
%	If $X,Y$ are two simplicial complexes such that $X\Searrow Y$, then $X$ is said to have the same \emph{strong homotopy type} as $Y$.
%\end{defi}
Barmak-Minian \cite{strong} proved the following theorems about strong collapsibility which will be used later.
\begin{thm}\label{barmak}
	\begin{enumerate}[label=\alph*)]
	\item \label{simultcoll}
		Let $L$ be a subcomplex of a complex $K$ such that every vertex of $K$ which is not in $L$ is dominated by some vertex in $L$. Then $K\Searrow L$.
	\item \label{strong2simple}
	If $X,Y$ are two simplicial complexes such that $X\Searrow Y$, then $X\searrow Y$.
	\item \label{scjoin}
	Given two simplicial complexes $X,Y$, their join $X\Join Y$ is strongly collapsible if and only if either $X$ or $Y$ is strongly collapsible.
	\end{enumerate}
\end{thm}

\section{Vocabulary}\label{vocab}
\subsection{The surfaces}
For each of the following topological 2-manifolds $S$, we choose a boundary component and mark $n\geq 1$ distinct points on it. These points are called \emph{vertices} and the set of all vertices is denoted by $P$. The portion of the boundary contained between two consecutive vertices is called an \emph{edge}. The set of all boundary edges is denoted by $\mathcal{E}$. Now we will define each surface individually.
%\PP{Do i need to define the edgeset?}
%%%%%%%%%%%%%%%%%%%%%%%% the surfaces %%%%%%%%%%%%%%%%%%%%%
\paragraph{Convex polygons.} For $n\geq 1$, we denote by $\poly n$, a closed disk with $n\geq 1$ marked points on its boundary. For $n\geq 3$, this surface, when endowed with a convex Euclidean metric, becomes the usual convex polygon. Also for $n\geq 3$, the surface $\poly n \setminus \mathcal{P}$, obtained by removing the marked points, admits a convex hyperbolic metric. In this metric, the vertices are on the ideal boundary of the hyperbolic plane and the edges are hyperbolic geodesics joining any two consecutive pair of ideal points. This surface is called an ideal polygon. 
%We will see in the next section that the combinatorics of their arcs and the topology of their arc complexes are identical. 
\begin{figure}[th!]
	\centering
	\includegraphics[width=12cm]{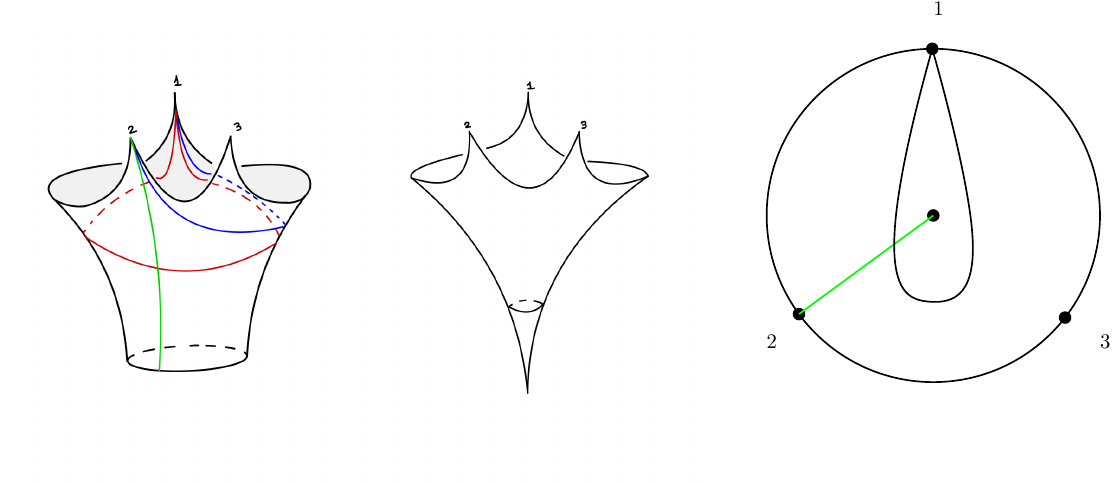}
	\caption{A crown with three vertices.}
	\label{surf}
\end{figure}
\paragraph{Orientable Crowns.} 
For $n\geq 1$, we denote by $\holed n$, the closed disk $\poly n$ with one marked point in its interior labeled as $0$. The topological surface $S=\holed n \setminus P \cup \{0\}$ is homeomorphic to an annulus. Let $\gamma$ be the generator of its fundamental group. This surface $S$ admits two types of convex hyperbolic metrics, depending on whether $[\gamma]$ is mapped to a parabolic or a hyperbolic element by the holonomy representation of the metric. See Fig. \ref{surf}. In either case, the marked points in the boundary are represented by ideal points in $\HPb$, like in the case of ideal polygons. We will refer to this topological surface $\holed n$ as a \emph{crown}.

\paragraph{Non-orientable crowns.}For $n\geq 1$, we denote by $\mob$ the M\"{o} strip with $n$ vertices on its boundary. We will refer to this surface as the \emph{non-orientable crown}. Once again, the surface obtained by removing the vertices admits a convex hyperbolic metric. See Fig. \ref{surf2}. The top left figure is the surface $\mob\setminus P$ endowed with a hyperbolic metric. The figure on the top right depicts the topological surface $\mob [3]$ — the the crossed circle at the centre represents a copy of the projective plane. The bottom figure represents the universal cover of this surface.
\begin{figure}
	\centering
	\includegraphics[width=15cm]{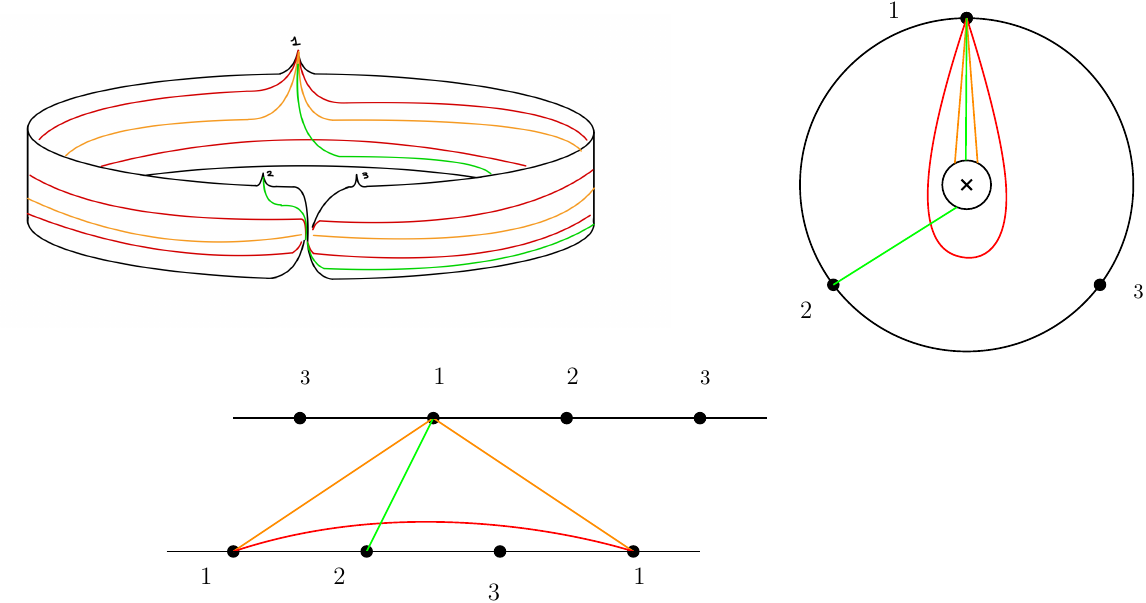}
	\caption{A non-orientable crown with three vertices.}
	\label{surf2}
\end{figure}

\paragraph{Integral strip.} For $m,n\geq 1$, let $P(m,n)$ be the quadrilateral whose corners are given by $(1,0)$, $(1,1),$ $(n,1)$ and $(m,0)$. Next we place $m-2$ vertices at the points $(2,0), \ldots, (m-1,0)$ on the $x$-axis and $n-2$ vertices $(2,1), \ldots, (n,1)$ on the $y=1$ line. These two sets are called the $x$ and $y$ vertices, respectively. In Fig. \ref{figstrip}, we have coloured the $x$-vertices in blue and the $y$-vertices in red. Topologically, this is a polygon $\poly{m+n}$ with a bicolouring. 

\subsection{Arcs and arc complexes}\label{acc}
Next we consider simple arcs on the surfaces defined in the previous section.
\paragraph{Simple arcs. } A \emph{simple arc} is an embedding $a: [0,1]\rightarrow S$ such that $a [0,1]\cap \partial S= \{a(0), a(1)\}$ where $a(0), a(1)$ are two vertices. A simple arc is said to be \emph{trivial} if it is homotopic, relative to its endpoints, to a vertex or an edge of the given surface. Since in this paper, we will consider only simple non-trivial arcs, therefore we will often omit the adjectives "non-trivial" and "simple" before the noun "arc". The homotopy class of an arc $a$ is denoted by $[a]$. Two homotopy classes of arcs are said to be disjoint if there exist two arcs in the respective equivalence classes which are disjoint. When endowed with a convex metric, Euclidean or hyperbolic, the minimal intersection between two homotopy classes is realised by geodesics. For example, in a convex polygon, these arcs are the diagonals.\\

Next we introduce the vocabulary for some specific types of arcs. 
%\begin{defi}\label{minimal}
%An arc of $S=\poly n,\holed n, \mob[n]$ is called \emph{minimal} if it separates a disk homeomorphic to triangle $\poly 3$ from the surface. In Fig. \ref{accrown} and \ref{acmob}, the blue arcs in the top right panels are minimal arcs.
%\end{defi}

\begin{defi}\label{maximal}
An arc of $S=\holed n, \mob[n]$ is called \emph{maximal} if both its endpoints coincide. 
\end{defi}

\begin{defi}\label{boundary}
In an orientable crown $\holed n$, any arc that does not have an endpoint on the vertex $0$, is called a \emph{boundary arc} or a b-arc in short. In a non-orientable crown $\mob$, any arc that decomposes the strip into one orientable and one non-orientable subsurface is called a boundary arc or b-arc. In either case, a non-boundary arc is called \emph{core arc} or simply \emph{c-arc}.
\end{defi}
\begin{rem}
In \cite{wilson}, Wilson introduced this terminology of arcs in the case of a non-orientable crown. There, a b-arc is short for \emph{bounding arc} and c-arc is short for \emph{cross-cap arc}. Since in this paper we are defining these two arcs for the orientable crown as well, we renamed the full-forms while keeping the shorthand. 
\end{rem}
\begin{rem}
A c-arc of a non-orientable crown $\mob[n]$ always intersects its one-sided core curve. Similarly, any c-arc of a one-holed polygon $\holed n$ intersects the two-sided core curve of $\holed n \smallsetminus \{0\}$.
\end{rem}
%A non-boundary arc is called an \emph{core arc} or a c-arc. For $i\in \intbra$, denote by $M_i$ the b-arc which has both its endpoints at the vertex labeled $i$. These are called \emph{maximal} b-arcs. A b-arc that separates a vertex from the rest of the surface is called a \emph{minimal arc}. All c-arcs have an endpoint at the internal vertex 0. For $i\in \intbra$, denote by $c_i$ the c-arc which has $i$ as its other endpoint. 

\begin{note}
\begin{enumerate}[label=\alph*)]
	\item Any c-arc of the surface $\holed n$ joins the internal point $0$ with some vertex $i\in \intbra$ in the boundary. So these arcs are going to be denoted by $c_i$. These are coloured green in the top panel of Fig. \ref{accrown}.
	\item For $i\in \intbra$ and $S=\holed n, \mob$,  we denote by $M_i$ the maximal b-arc with both its endpoints at the vertex labeled $i$. These are coloured in red in Figs. \ref{accrown}, \ref{acmob}.
	\item For $i\in \intbra$ and $S= \mob$, we denote by $L_i$ the maximal c-arc with both its endpoints at the vertex labeled $i$. These are coloured in orange in Fig.\ref{acmob}.
	\item Given two vertices $i,j$ of $\mob$ such that $1\leq i< j\leq n$, we denote by $(i,j)$ the unique c-arc joining these two vertices. These are coloured in green in Fig. \ref{acmob}.
\item An arc of $S=\poly n,\holed n, \mob[n]$ is called \emph{minimal} if it separates a disk homeomorphic to triangle $\poly 3$ from the surface. In Fig. \ref{accrown} and \ref{acmob}, the blue arcs in the top right panels are minimal arcs.
\end{enumerate}
\end{note}

Next we introduce the arc complex of every surface.
	
Let $S$ be any of the three surfaces defined above. Let $\mathcal X$ be the set of homotopy classes of all non-trivial arcs. Let $\mathcal C\subset \mathcal X$ be the subset containing the homotopy classes of c-arcs. Finally, let $\mathcal B\subset \mathcal X$ be the subset containing the homotopy classes of b-arcs.
 %%%%% arc complex %%%%%	
 \begin{figure}[th!]
 	\centering
 	\includegraphics[width=15cm]{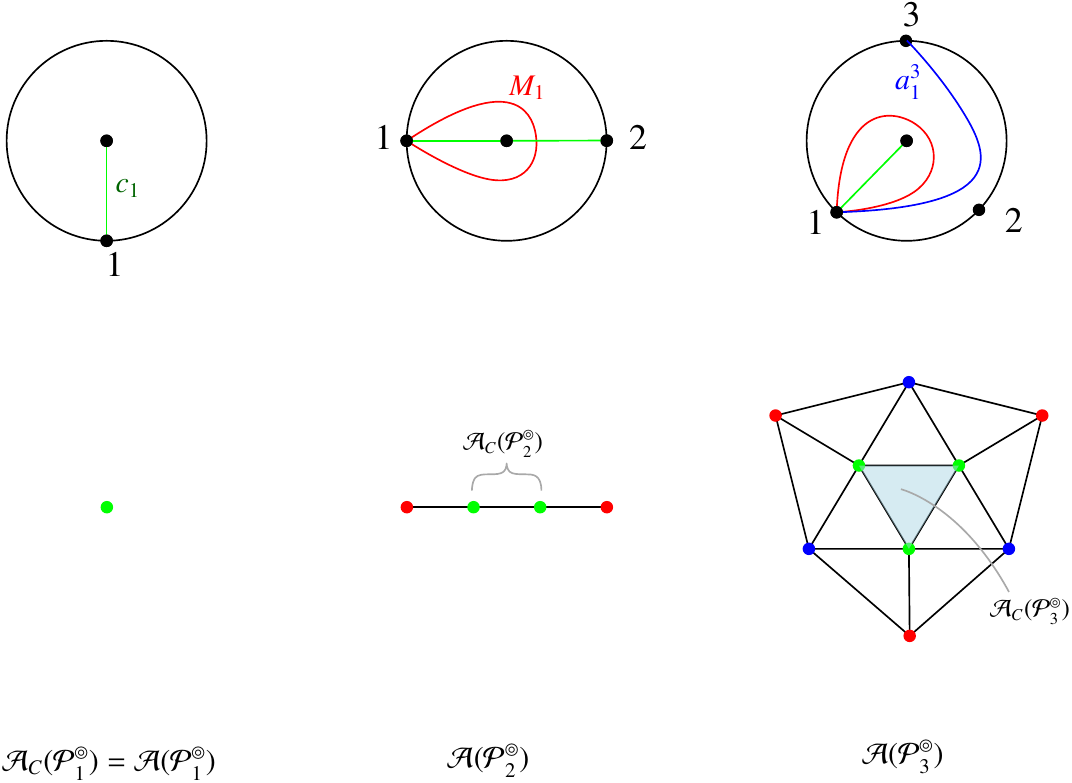}
 	\caption{The arc complexes of $\holed 1, \holed 2$ and $\holed 3$}
 	\label{accrown}
 \end{figure}
	\begin{defi}
	The arc complex $\ack K S$ generated by $K\subset \mathcal X$ of the surface $S$ is a simplicial complex whose 
	\begin{itemize}
	\item 0-simplices are given by the elements of $K$;
	\item for $k\geq 1$, every $k$-simplex is given by a $(k+1)$-tuple of pairwise disjoint and distinct homotopy classes.
	\end{itemize}
	\end{defi}
\begin{figure}[th!]
	\centering
	\includegraphics[width=15cm]{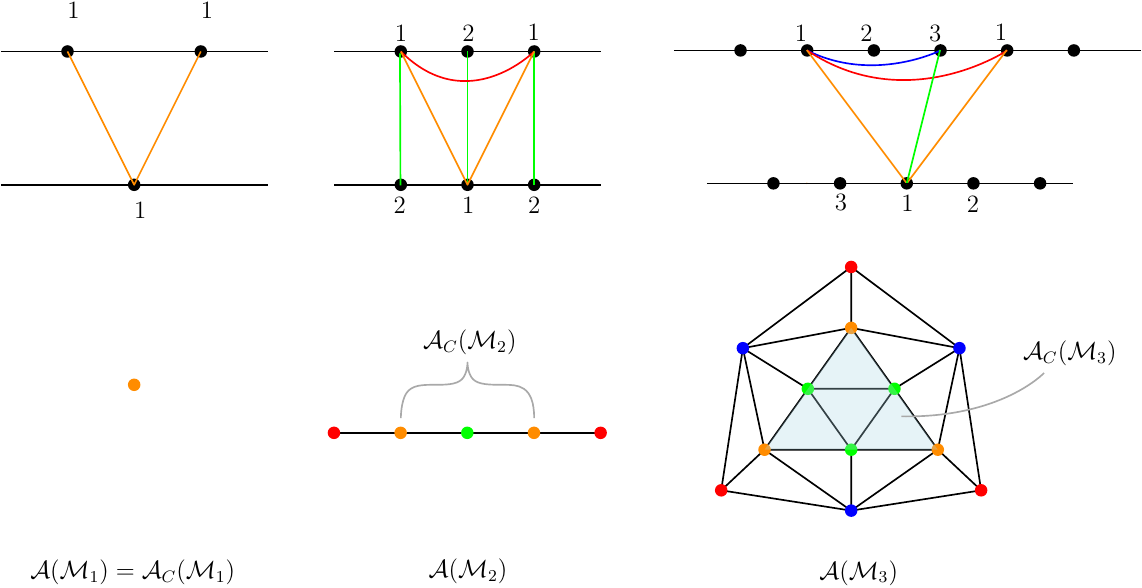}
	\caption{The arc complexes of $\mob[1], \mob[2]$ and $\mob [3]$}
	\label{acmob}
\end{figure}

The complex $\ack XS$ is called the full arc complex and is denoted simply by $\ac S$. In Fig. \ref{accrown} and \ref{acmob}, the full arc complexes of the two surfaces $\holed n$ and $\mob$ have been illustrated, for $n=1,2,3$.
\begin{defi} For $S=\holed n, \mob$, the complexes $\ack {\mathcal C}S$ and $\ack {\mathcal B}S$ are called the internal and the boundary arc complexes. A simplex of $\ack {\mathcal C}S$  (resp. $\ack {\mathcal B}S$) is called a \emph{core} (resp. \emph{boundary}) simplex.
\end{defi}
\begin{defi}
For $S=P(m,n)$, let $K$ be the set of all arcs that have one blue and one red endpoint. Then the arc complex $\ack K {P(m,n)}$ will be denoted as $\acs{m,n}$. For $i\in \intbra[1,m]$ and $j\in \intbra[1,n]$, an arc joining the $i$-th blue vertex with the $j$-th red vertex will be denoted by $(i,j)$.
\end{defi}
\begin{figure}[th]
	\centering
	\includegraphics[width=12cm]{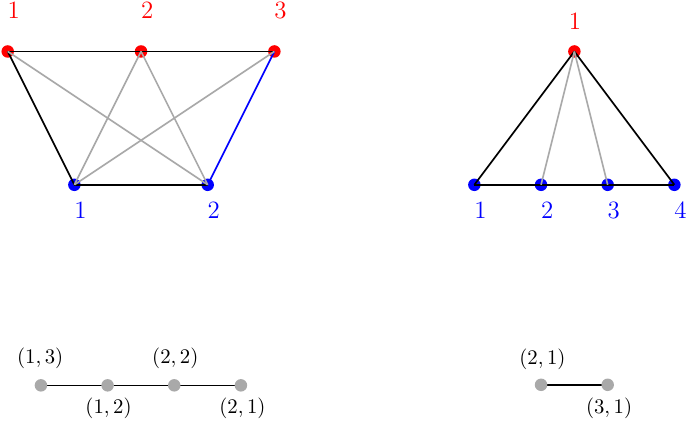}
	\caption{The arc complexes $\mathcal A_{2,3}$, $\mathcal A_{4,1}$ are strongly collapsible.}
	\label{intstripbase}
\end{figure}

\begin{Ex}
The internal arc complex $\ack {\mathcal C}{\holed n}$ of a crown $\holed n$ is an $(n-1)$-simplex whose vertices are given by $[c_i]$, for $i\in \intbra$.  
\end{Ex}
\begin{rem}\label{embed}
If the surface $S$ embeds into the surface $S'$ and $K\subset K'$, then the complex $\ack K S$ embeds into the bigger complex $\ack {K'} {S'}$. 
\end{rem}
\begin{rem}
The boundary arc complexes of $\holed n$ and $\mob$ are isomorphic. 
\end{rem}
\begin{rem}
Let $a,b$ be two non-homotopic arcs. Then the 0-simplex $[a]$ is vertex-dominated by the 0-simplex $[b]$ if and only if any arc disjoint from $a$ is also be disjoint from $b$. 
\end{rem}

	\begin{figure}
		\centering
		\includegraphics[width=15cm]{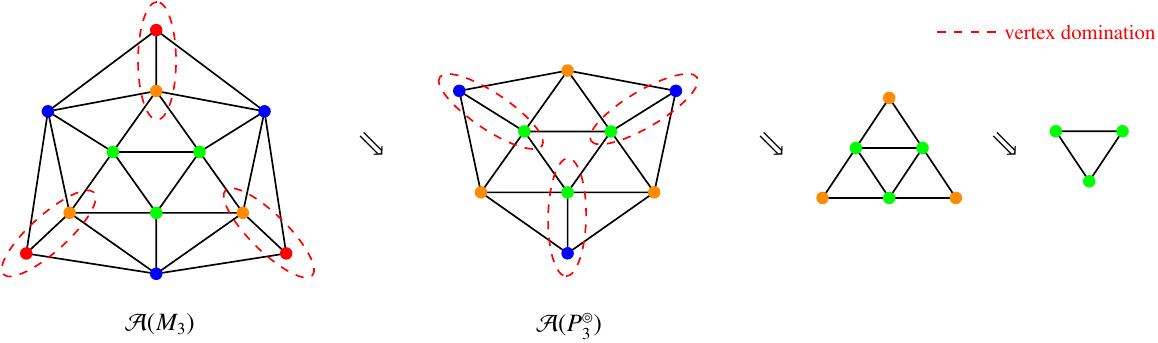}
		\caption{A coincidence in dimension 2}
		\label{collapse1}
	\end{figure}
\subsection{Tiles}
	\begin{defi}
		Let $\sigma$ be a simplex of the arc complex $\ac S$ of a surface $S$. Let $a_1,\ldots, a_k$ be pairwise disjoint and distinct arcs such that the 0-simplices of $\sigma$ are given by $\sigma^{(0)}=\bigcup\limits_{i=1}^k\{ [a_i]  \}$.
		Then we say that $\sigma$ decomposes the surface $S$ into the \emph{tiles} $\del_1,\ldots, \del_p$ if $$S\setminus \bigcup\limits_{i=1}^k a_i=\del_1\sqcup\ldots\sqcup\del_p,$$ where $\del_1,\ldots, \del_p$ are the connected components. 
\end{defi} 

\emph{Notation: }A b-arc of a (resp. non-orientable) crown $\holed n$ (resp.\ $\mob$) decomposes the surface into two tiles: one homeomorphic to a polygon and one homeomorphic to a (resp. non-orientable) crown with at most $n+1$ vertices. Given two vertices $i,j$ such that $1\leq i\leq j \leq n$, there are at most two b-arcs joining them. When $j-i>1$, let $a_i^j$ be the b-arc  whose polygonal tile is homeomorphic to $\poly {j-i+1}$, containing the vertices $i, i+1,\ldots,j$; its non-polygonal tile is homeomorphic to $\poly{n-j+i+1}$ (resp.$\mob[n-j+i+1]$), containing the vertices $j+1,\ldots, n, 1\ldots, i$. When $j-i<n-1$, we denote the second b-arc by $a_j^i$. See Fig. \ref{notation}.
	\begin{figure}[th]
	\centering
	\includegraphics[width=12cm]{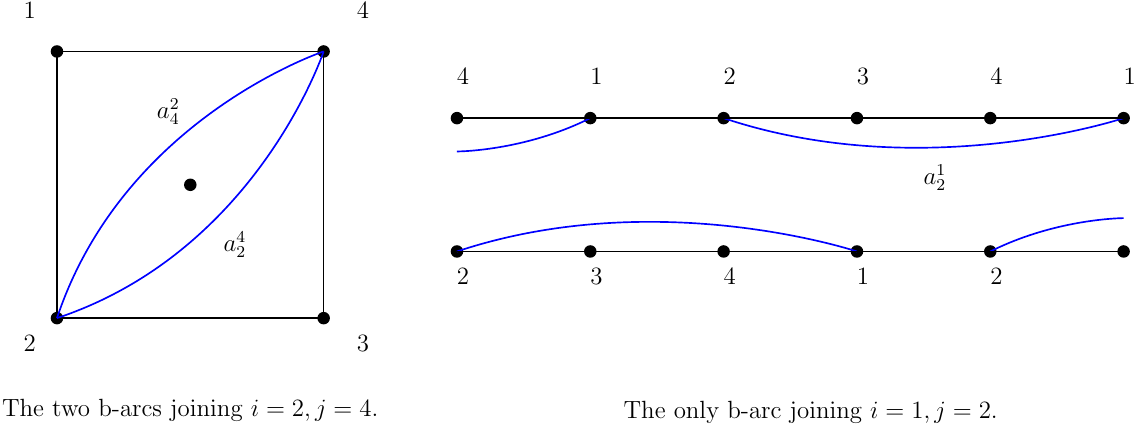}
	\caption{}
	\label{notation}
\end{figure}
%\begin{obs}
%The arc $a_{i_0j_0}$ is disjoint to
%\begin{itemize}
%	\item b-arcs
%	\begin{itemize}
%		\item maximal: $a_{ii}$ or $M_i$, for $1\leq i\leq i_0$ and $j_0\leq i\leq n$. 
%		\item non-maximal: 
%		\begin{itemize}
%			\item inside the polygon: $a_{ij}$, for $i_0\leq i<i+2 <j\leq j_0$
%			\item inside the Mobius strip: $a_{ij}, a_{ji}$, for $1\leq i\leq j_0$ and $j_0\leq i\leq n$.
%		\end{itemize}
%	\end{itemize}
%	\item c-arcs 
%	\begin{itemize}
%		\item maximal: $L_i$ for $1\leq i\leq i_0$ and $j_0\leq i\leq n$. 
%		\item non-maximal: $(i, j)$ $1\leq i\leq j_0$ and $j_0\leq i\leq n$.
%	\end{itemize}
%\end{itemize}
%\end{obs}

A simplex of $\ac S$ is called a \emph{triangulation} if it decomposes the surface into tiles homeomorphic to triangles, $\poly 3$. These are the maximal simplices of $\ac S$. 

\begin{Ex}\label{fanmob}
	For every $i\leq j\leq k\in\intbra$, let $\fan{i}{j,k}$ be the simplex of $\ac{\mob}$ whose 0-simplices are given by the isotopy classes of the c-arcs $(i,j), (i,j+1),\ldots, (i,k)$. When $i=j=k$, this simplex is maximal in $\ac\mob $ and is called a \emph{fan triangulation} based at the vertex $i$, denoted by $\ffan[i]$. See Fig. \ref{collapse1}.
\end{Ex}
\begin{rem}
	Any triangulation of the surfaces $\holed n$ and $\mob$ always has at least one $c$-arc. In fact, it is possible to triangulate the surface using only $c$-arcs. 
\end{rem}
Wilson \cite{wilson} proved the following theorem about the shellability of the arc complexes of the two type of crowns:
\begin{thm}[Wilson]
	\begin{enumerate}[label=\alph*)]
		\item For $n\geq 3$, the full arc complex $\ac{\poly n}$ of a convex $n$-gon is a shellable sphere of dimension $n-4$.
		%\item For $n\geq 1$, the boundary arc complex $\ack{\mathcal{B}}{\holed n}$ of an orientable crown is a shellable sphere of dimension $n-2$.
		\item For $n\geq 1$, the internal arc complex $\ack{\mathcal{B}}{\mob}$ of a non-orientable crown $\mob$ is a shellable pseudo-manifold of dimension $n-1$.
		\end{enumerate}
\end{thm}

\paragraph{Botany.}When $S=\holed n$, exactly one tile in the tileset of a boundary simplex contains the vertex $0$ in its interior. When $S=\mob$, exactly one tile in the tileset of a boundary simplex is non-orientable. In both the cases, this unique tile is called the \emph{trunk}. The boundary of the trunk consists of arcs of $\sigma$ called \emph{branch arcs}. The simplex generated by the branch arcs of $\sigma$ is called the \emph{stem} and is denoted by $\Stem$. The boundary may also contain some edges of the original surface -- these are called \emph{roots}. See Fig. \ref{botany}.

\begin{figure}[th]
	\centering
	\includegraphics[width=12cm]{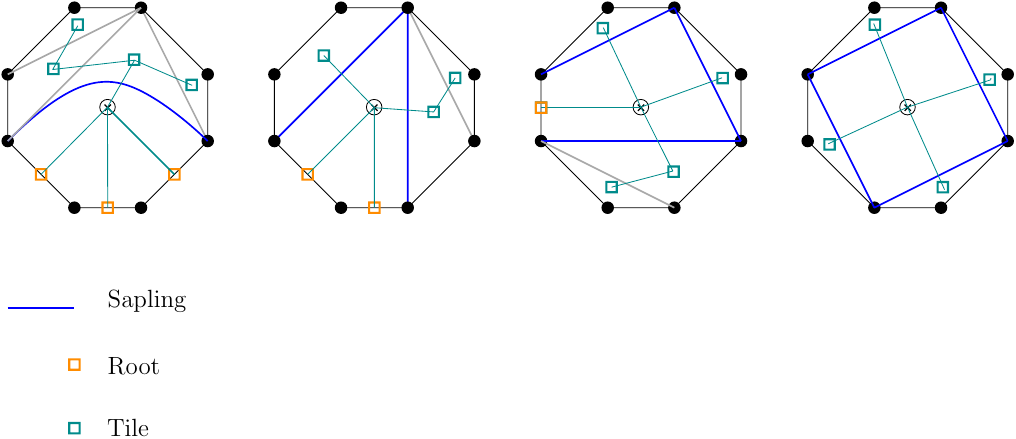}
	\caption{Boundary simplices of degree 4 in $\ac{\mob[8]}$}
	\label{botany}
\end{figure}
\begin{defi}
	Let $S=\holed n, \mob$. For any boundary simplex $\sigma \in \ac S$, the \emph{dual tree} $\dtree$ is a graph defined in the following way:
\begin{itemize}
	\item The vertices are given by the tiles of $\sigma$ and the roots, if any;
	\item Every root is joined by an edge to the trunk. Any two non-root vertices are joined by an edge if and only if the corresponding tiles share an arc of $\sigma$.
\end{itemize}
\end{defi}
The dual tree is a tree in the sense of graphs. The vertices of the dual tree that are non-roots and have degree one are called \emph{leaves}. The degree of the trunk is the sum of the total number of roots and branches. A boundary simplex $\sigma$ is said to have \emph{degree} $d\geq 1$ if the degree of the trunk of $\sigma$ in the dual tree is $d$.  A boundary simplex, all of whose branches are leaves in its dual tree, is called a \emph{sapling} and is denoted by $\sap$. Finally, let $\eta$ be a simplex of $\ac S$ with at least one b-arc and at least one c-arc. Then we denote by $\eta_b$ the subsimplex of $\eta$ spanned only by $b$-arcs. The degree of $\eta$ is then defined as the degree of the boundary simplex $\eta_b$. See Fig.\ref{botany} for boundary simplices of degree 4 of the arc complex of a non-orientable crown. 
\begin{Ex}
	A sapling of degree one is a maximal arc and a sapling of degree $n-1$ is a minimal arc. Any 0-simplex of $\ac S$ is a sapling with only one branch arc. %Given any boundary simplex $\tau$ containing a maximal arc $M_i$ for some $i\in\intbra$, we always have $\Stem[\tau]=[M_i]$. 
\end{Ex}	

%Vocabulary for certain special arcs:
%	\begin{defi}
%For $i\in\llbracket{1,n}\rrbracket$,
%\begin{itemize}
%	\item for $p=0,\ldots, n-1$, let $a^p_{i}$ be the arc separating the marked points $i,i+1,\ldots,i+p$ from the rest of the surface.
%	\item denote by $M_i$ the arc $a^{n-1}_i$ which has both its endpoints at $i$. These are called \emph{maximal} arcs.
%	\item denote by $m_i$ the arc $a^{0}_i$. These are called \emph{minimal} arcs.
%	\item let $c_i$ be the arc joining the $i$-th marked point with $x_0$;
%\end{itemize}
%	\end{defi}

\section{Collapsibility of the arc complexes}\label{proofs}
In this section, we will give the proofs of our main results on collapsibility.
\subsection{The arc complex of a crown}
Firstly, we prove the strong collapsibility of the arc complex of a crown. 
%%%%%%%%% alternate notation: using i,j %%%%%%%%%% 
%For $1\leq i < j\leq n$,  there are exactly two boundary arcs connecting the $i$-th and the $j$-th vertices of the surface $\holed n$. We will denote by $a_i^j$ that arc among these two that separates a tile homeomorphic to $\holed{j-i+1}$ on its right, from the surface. The other tile is then homeomorphic to $\poly{n-j+i+1}$. The other arc is denoted by $a_j^i$.

%%%%%%%%% alternate notation: using i,k %%%%%%%%%%
%For $i\in \intbra$ and $k\in \intbra [0,n-2]$, there are exactly two boundary arcs connecting the two vertices $i$ and $i+k$. We will denote by $a_i^{i+k}$ that arc among these two that decomposes the surface into two tiles homeomorphic to $\holed {k+1}$ and $\poly {n-k+1}$. The second arc is denoted by $a^i_{i+k}$. 
%%%%%%%%% alternate notation using i,k  %%%%%%%%%%

%%%%%%%%% alternate notation using only k %%%%%%%%
\begin{thm}\label{strcollhole}
	For $n\geq 1$, the full arc complex $\ac{\holed n}$ of a crown $\holed n$ is strongly collapsible.
\end{thm}	
\begin{proof}
	We will prove that the complex $\ac{\holed n}$ strongly collapses onto $\ack C{\holed n}$, which is an $(n-1)$-simplex. 
	%Let $Y^0=\ac{\holed n}$, $Y^{n-1}=\ack C{\holed n}$. For $k=1,\ldots, n-2$, define $$Y^k:=Y^{k-1}\smallsetminus \bigcup\limits_{a\in \mathcal{B}_k}\{[a]\}$$
	For $k\in \intbra$, let $\mathcal{B}_k$ be the set of all arcs that decompose the surface into two tiles homeomorphic to $\holed {k}$ and $\poly {n-k+2}$, respectively. Define the following subcomplexes of $\ac{\holed n}$:
	\[ 
	Y^k:=\left\{
	\begin{array}{ll}
	\ac{\holed n}, & \text{ when } k=0,\\
	 Y^{k-1}\smallsetminus \{[a] \mid a\in \mathcal{B}_k
 \}& \text{ when } k \in \intbra[1,n-1] \\
	\end{array}
	\right.
	\]
	\begin{claim}
		For $k=0,\ldots, n-2$, $Y^k$ strongly collapses onto $Y^{k+1}$.
	\end{claim}
	We prove this claim by induction on $k$. \\
	\emph{Base step:} The set $\mathcal{B}_1$ comprises of all the maximal arcs of $\holed n$. %We need to show that $Y_0=\ac{\holed n}$ strongly collapses onto $Y_1=\ac{\holed n} \smallsetminus \cur{[M_i] \mid i\in \intbra}$. 
	Indeed, any maximal arc $M_i$ decomposes the surface $\holed n$ into a tile homeomorphic to $\poly{n+1}$ and a tile homeomorphic to $\holed 1$, containing the vertex $0$ in its interior and the vertex $i$ in its boundary. Hence the link of the 0-simplex $[M_i]$ in the full arc complex is given by 
	\begin{align*} 
	\Link {[M_i]}{Y_0} &= \ac{\poly{n+1}} \Join \ac {\holed 1}\\
	& = \s{n-3}\Join [c_i].
	\end{align*}
	This shows that for every $i\in \intbra$, the 0-simplex $[M_i]$ is vertex-dominated by $[c_i]$. %Since any two maximal arcs $M_i,M_j$, with $i\neq j$ always intersect, a triangulation cannot contain both. So we have $$\Link {[M_i]}{Y^0\smallsetminus \{[M_j]\}}=\Link {[M_i]}{Y_0}.$$ Hence $Y^0\smallsetminus \{[M_j] \Searrow  Y^{k-1}\smallsetminus \{[M_i]\}\cap Y^{k-1}\smallsetminus \{[M_j]\}$. Repeating this argument, 
	Using Theorem \ref{barmak}\ref{simultcoll}, we get that $Y^0\Searrow Y^1$. This finishes the base step $k=0$.\\
	\emph{Induction step:} Suppose that for $k'=0,\ldots, k-1$, we have $Y^{k'}\Searrow Y^{k'+1}$. We need to show that $Y^{k} \Searrow Y^{k+1}$. %We show that for every $i\in \intbra $ that 0-simplex $[a_i^{i+k}]$ is vertex-dominated by $[c_i]$ as well as $[c_{i+k}]$ in $Y^k \smallsetminus \{{[a_1^{1+k}], \ldots, [a_{i-1}^{i-1+k}]}\}$. 
	Let $a\in \mathcal{B}_{k+1}$ have endpoints on the vertices $i\neq j$. Without loss of generality, suppose that the tile homeomorphic to $\holed {k+1}$ also contains the vertex $i+1$. Then following Notation \ref{notation}, we have that $a=a_j^i$. We claim that the 0-simplex $[a]$ is vertex-dominated by the c-arc $[c_i]$. Any triangulation containing the arc $a$ either contains the arc $c_{i}$ or the b-arc $\aij {j}{i+1}$ joining $i+1$ and $j$ which intersects $c_i$. See Fig. \ref{crownacsc}.  But the arc $\aij {j}{i+1}$ decomposes the surface into $\holed k$ and $\poly {n-k+2}$. So it lies in $\mathcal{B}_k$. The 0-simplex $[\aij {j}{i+1} ]$ is absent from $Y^k$, by induction hypothesis. Thus every arc in $\mathcal{B}_{k+1}$ is vertex dominated in $Y^k$. %and using a similar argument,  by $[c_{1+k}]$.
	%When $k<\lfloor{\frac{n}{2}}\rfloor$, any two arcs in $\mathcal{B}_{k+1}$ intersects. 
Using Theorem \ref{barmak}\ref{simultcoll}, as in the base step, we conclude that $Y^k \Searrow Y^{k+1}$.
% on to $Y^{k+1}$ by removing these vertex-dominated 0-simplices in any order. Now suppose that $k\geq \floor{\frac{n}{2}}$. Then there exists $a,b \in \mathcal{B}_{k+1}$ that are disjoint. So their corresponding 0-simplices are connected by an edge in $Y^k$. Let $L_a$ be the subcomplex of $Y^k$ such that $\Link{[a]}{ Y^k}= [c_a]\Join L_a$. Then we have that
	%\begin{align*}
%	\Link { [a]} {Y^k\smallsetminus { [b] }}&=\Link{[a]}{ Y^k}\smallsetminus  { [b]}\\
%	&=[c_{a}]\Join (L \smallsetminus [b]),
%	\end{align*}
%	which implies that $[a]$ remains vertex-dominated by $[c_a]$ in $ Y^k\smallsetminus { [b] }$. So once again we can strongly collapse in any order. 
Since any boundary arc lies in $\mathcal{B}_k$ for some $k$, we get that $Y^{k-1}= \ack C {\holed n}$. This concludes the proof.
	%Hence, $Y^k \Searrow Y^k\smallsetminus\{ [a^{1+k}_{1}], [a_{1+k}^{1}]  \}.$ Now suppose that for $i'=1,\ldots, i-1,$ the complex $Y^k\smallsetminus\bigcup\limits_{j=1}^{i'-1}\{ [a^{j+k}_{j}], [a_{j+k}^{j}]\} \Searrow Y^k\smallsetminus\bigcup\limits_{j=1}^{i'}\{ [a^{j+k}_{j}], [a_{j+k}^{j}]\}$. 
	\end{proof}
\begin{figure}
	\centering
	\includegraphics[width=5cm]{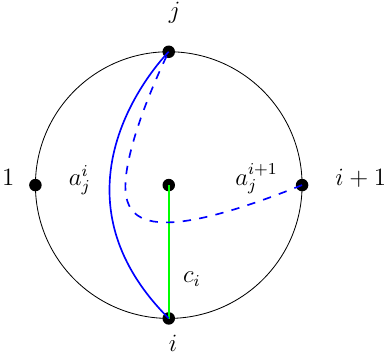}
	\caption{An arc $a_j^i\in \mathcal{B}_{k+1}$ is vertex-dominated by $c_i$ in $Y^k$.}
	\label{crownacsc}
\end{figure}
\begin{rem}
Using a similar argument, it can be shown that the arc in $a_j^i$ is vertex-dominated by $c_j$ in $Y^k$.
\end{rem}
%%%% coll ball
\begin{cor}\label{acrownball}
For $n\geq1$, the arc complex of an orientable crown $\holed n$ is a combinatorial ball of dimension $n-1$. 
\end{cor}
\begin{proof}
%	Firstly we show that the complex is pure. Let $\sigma$ be any simplex of $\ac{\holed n}$. Suppose that $\sigma$ contains a c-arc $c$. Cutting the surface along this arc we get a surface homeomorphic to the convex polygon $\poly{n+2}$. A triangulation of $\holed n$ will give a triangulation of this polygon. We know that the complex  $\ac{\poly{n+2}}$ is a pseudo-manifold of dimension $n-2$, so any maximal simplex $\eta$ of $\ac{\poly{n+2}}$ containing $\sigma$ is of dimension $n-2$. So the simplex $\sigma$ of $\ac{\holed n}$ is contained in the maximal simplex $[c]\Join \eta$, which is of dimension $n-1$. Since any maximal simplex containing $\sigma$ is of the above form, we get that the complex $\ac{\holed n}$ is pure, of dimension $n-1$. Next suppose that $\sigma$ contains no internal arc. Then there exists a vertex $i$ on the boundary of $\holed n$ which is contained in the same tile as $0$ in $\tile$. Then we argue as in the previous case with the simplex $\sigma\cup \{[c_i]\}$. \\ 

%The strongly connectedness of the dual graph of $\ac{\holed n}$ follows from the work of Parlier-Pournin \cite{pointi}.\\
We know already that the arc complex is pure of dimension $n-1$ and strongly connected. 
Firstly, we show that every $(n-2)$-simplex $\sigma$ is contained in at most two $(n-1)$-simplices. Suppose that $\sigma $  is a boundary simplex.  Then it decomposes the surface into triangles and a crown $\mob[1]$ with only one vertex, say $i\in \intbra$. The only way to triangulate this tile is to take the c-arc $c_i$. So $\sigma$ is contained in the unique maximal simplex $\sigma\cup \{[c_i]\}$. Next, we suppose that $\sigma$ contains an internal arc $c$. Cutting the surface along this arc we get a surface homeomorphic to the convex $n+2$-gon. The restriction of $\sigma$ on this polygon gives a $(n-3)$-simplex, say $\sigma'$, of $\ac{\poly {n+2}}$. Since the latter is a $(n-2)$-pseudo-manifold, $\sigma'$ is contained in at exactly two $(n-2)$-simplices of $\ac{\poly{n+1}}$. Taking the join of these two maximal simplices with $[c]$ we get that $\sigma$ is contained in exactly two $(n-1)$-simplices of $\ac {\holed n}$. So we get that the simlicial complex $\ac {\holed n}$ is a pseudo-manifold. \\
Now, we prove that the arc complex is a combinatorial ball of dimension $n-1$, by induction. From Fig. \ref{accrown}, we know that the statement is true for $n=1,2,3$. Suppose that the statement is true for $n'=1,\ldots, n-1$. 
Let $\al$ be a c-arc. Then it decomposes the surfaces into a $\poly {n+2}$. Its link is then given by $$\Link{[\al]} {\ac{\holed n}}=\ac {\poly{n+2}}= \s{n-2}.$$
Next, suppose that $\al$ is a b-arc of the form $a_i^j$ with $j-i\geq2$. Then it decomposes the surface into a polygon $\poly{j-i+1}$ and a crown $\holed{n-j+i +1}$. Then its link is given by
\begin{align*}
\Link{[\al]} {\ac{\holed n}}&=\ac{\poly{j-i+1}}\Join \ac{\holed{n-j+i+1}}\\
&=\s{j-i+1-4}\Join \ball{n-j+i+1-1 }\\
&=\ball{n-2}.
\end{align*}
So we get that $\ac{\holed n}$ is a combinatorial manifold of dimension $n-1$. Finally, using Whitehead's Theorem \ref{whitehead} and our Theorem \ref{strcollhole}, we get that the full arc complex is a combinatorial $n-1$-ball.
\end{proof}

	\subsection{The inner arc complex of a non-orientable crown}

	Firstly, we prove that the inner arc complex of a non-orientable crown is collapsible.
	\begin{thm}\label{inner mobius}
		For $n\geq 1$, the inner arc complex $\ack C{\mob}$ of a non-orientable crown $\mob$ is strongly collapsible.
	\end{thm}
	\begin{proof}
We prove the statement by induction on the number of vertices $n$.
When $n=1$, there is only one (homotopy class of) arc, namely the maximal c-arc $L_1$. So the inner arc complex is the 0-simplex $[L_i]$, which is strongly collapsible.

Next we suppose that the statement holds for all $n'\in \intbra$. Given the surface $\mob$, without loss of generality, we add a new vertex labeled $n+1$ between the vertices 1 and $n$ to get the surface $\mob [n+1]$. The idea is to strongly collapse the complex $\ack C{\mob[n+1]}$ onto the complex $\ack C {\mob } $ by removing all new arcs that arise in $\mob[n+1]$ due to the addition of the new vertex. All these new arcs have an endpoint at the vertex $n+1$.

Let $a_i:= [(i,n+1)]$, for $i\in \intbra[1,n+1]$. In particular, $a_{n+1}=[L_{n+1}]$.
We show that 
\begin{align*}
\across[n+1]  &\Searrow \across[n+1]   \smallsetminus \{a_{n+1} \} \\
&\Searrow \across[n+1]   \smallsetminus \{a_{n+1}, a_n \} \\
& \Searrow \ldots \\
&\Searrow  \across[n+1]   \smallsetminus \{a_{n+1}, a_n,\ldots, a_2, a_1 \} \\
&=\across[n] .
\end{align*}
%collapse long arc
		
\begin{step} 
$\across[n+1] \Searrow  \across[n+1]   \smallsetminus \{a_{n+1} \}$: the 0-simplex $[L_{n+1}]$ is contained in a unique maximal simplex, namely, $\ffan [n+1]$. So it is vertex-dominated by $[a_i]$ for any $i\in \intbra$. See Fig. \ref{collapse1}.
\end{step}
%basestep
%\begin{step}
%	$\fdel{a_{n+1}}{\across[n+1]     } \searrow \fdel{a_{n}} {\fdel{ a_{n+1}}{ \across[n_0+1]  }}$: the arc $(n, n+1)$ divides the surface into an integral strip $P(n+1, 2)$. The $x$-vertices of this strip are $n, n+1$ and the $y$-vertices are $ n+1, 1,\ldots, n$.  Since the 0-simplex $[L_{n+1}]$ has been collapsed already, none of the  0-simplices in the complex $\fdel{a_{n+1}}{\across[n+1] } $ correspond to an arc which has an endpoint on the $y$-vertex marked as $n+1$. So we get that $$\Link{a_{n}  }{ \fdel{ a_{n+1}}{ \across[n+1]  }  } = \ac {P(n,2)}\Join [(1, n)],$$ which is collapsible.
%\end{step}
%inductionstep
\begin{figure}[ht]
	\centering
	\includegraphics[width=10cm]{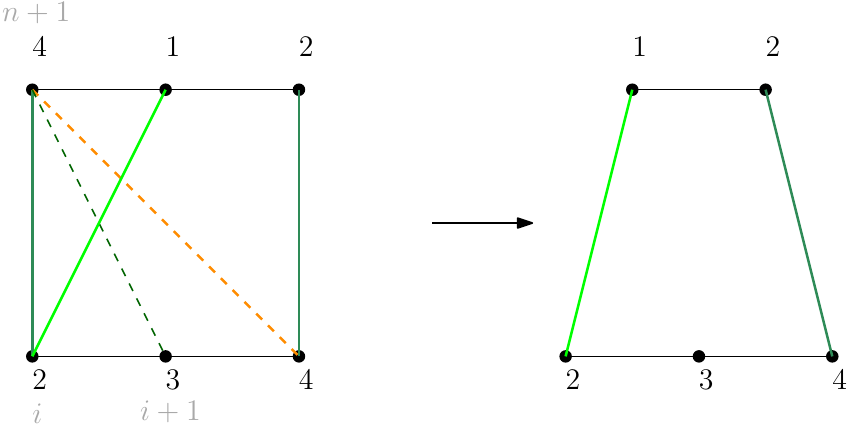}
	\caption{Step \ref{inductionstep} of Theorem \ref{inner mobius}}
	\label{induction}
\end{figure}

\begin{step} \label{inductionstep}
	Suppose that for $i'=n+1, n, \ldots, i+1$, the 0-simplex $a_{i'}$ is vertex-dominated in $\across[n+1]   \smallsetminus \{a_{n+1}, a_n, \ldots, a_{i'+1} \}.$ Then we need to show that the 0-simplex $[a_{i}]$ is vertex-dominated in $\across[n+1]   \smallsetminus \{a_{n+1}, a_n, \ldots, a_{i+1} \}.$
	%$$\fdel{a_{n_0-b}}{\ldots, {\fdel{ a_{n_0+1}}{X}}}  \searrow \fdel{a_{n_0-b+1}}{ \fdel{a_{n_0-b}}{\ldots, \fdel{a_{n_0+1}}{ X}}}.$$
See Fig. \ref{induction}. The arc $a_{i}$ decomposes the Möbius strip into an integral strip whose $x$-vertices are $n+1, 1,\ldots, i$ and $y$-vertices are $i, i+1,\ldots, n+1$. By induction hypothesis, none of the 0-simplices of $ \across[n+1]   \smallsetminus \{a_{n+1}, a_n, \ldots, a_{i} \}$ corresponds to an arc that has an endpoint on the $y$-vertex labeled $n+1$. So we have that
$$\Link{a_{i}  }{ \across[n+1]   \smallsetminus \{a_{n+1}, a_n, \ldots, a_{i+1} \}   } = \acs {i,n-i+2}\Join [(1, i)].$$ So $a_{i}$ is vertex-dominated by $[(1, i)]$. This finishes the induction on $i'$. By induction hypothesis on $n$, the complex $\across[n]$ is strongly collapsible, which implies that the complex $\across[n+1]$ is strongly collapsible. This finishes the induction on $n$. 
	\end{step}

	\end{proof}

\begin{figure}[ht]
	\centering
	\includegraphics[width=10cm]{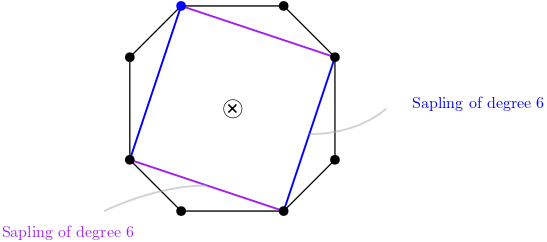}
	\caption{Two disjoint saplings of degree 6 are together forming a sapling of degree 4}
	\label{disjsap}
\end{figure}
\begin{thm}\label{mobboundarycollapse}
For all $n\geq 1$, the simplicial complex $\ac\mob$ simplicially collapses onto $\across$.
\end{thm}	
\begin{proof}
	Let $\mathrm{Sap}^d$ be the set of all saplings of degree $d$. 
	%Let $Y^0=\ac{\mob}$. For $d=1,\ldots, n-1$, let $Y^d:=\bigcap\limits_{\sigma_s\in \mathrm{Sap}^d}\fdel{\sap}{Y^{d-1}}$. 
	Define
		\[ 
	Y^d:=\left\{
	\begin{array}{ll}
	\ac{\mob}, & \text{ when } d=0,\\
	\bigcap \limits_{\sigma_s\in \mathrm{Sap}^d} \fdel{\sap}{Y^{d-1}}& \text{ when } d \in \intbra[1,n-1] \\
	\end{array}
	\right.
	\]
	\begin{claim}
		For $d=0,\ldots, n-2$, $Y^{d}$ collapses onto $Y^{d+1}$.
	\end{claim}
	We prove this claim by induction on $d$. \\
	\emph{Base step:} We need to show that $Y^0$ collapses onto $Y^1$. The saplings of degree 1 correspond to the maximal b-arcs $M_i$, for $i\in\intbra$. Any such arc decomposes the surface $\mob [n]$ into two tiles homeomorphic to $\poly{n+1}$ and $\mob [1]$, containing the vertex $i$ in its boundary. Hence the link of the 0-simplex $M_i$ in $Y_0$ is given by 
	\begin{align*} 
	\Link {[M_i]}{\ac{\mob }} &= \ac{\poly{n+1}} \Join \ac {\mob[1]}\\
	& = \s{n-3}\Join [L_i].
	\end{align*}
	which shows that for every $i\in \intbra$ the 0-simplex $[M_i]$ is vertex-dominated in $Y^0$. Using Lemma \ref{barmak}\ref{simultcoll}, we get that $Y^0\Searrow Y^1$. Finally, from Theorem \ref{barmak}\ref{strong2simple}, we get that $Y^0\searrow Y^1$. This finishes the base step $d=0$.

	\emph{Induction step:} Suppose that for all $d'=1,\ldots, d-1$, the complex $Y^{d}$ collapses onto the complex $Y^{d'+1}$. We need to show that $Y^{d}$ collapses onto $Y^{d+1}$.%So for all $d\in\intbra[1,d_0-1]$, the boundary simplices containing a sapling of degree $d$ have been faced-deleted from $Y^0$. 
	%Since the stem of a simplex is unique, we can face-delete the saplings of degree $d_0+1$ in the any order. 
	
	Let $\sigma_s$ be any sapling of degree $d+1$. Then the arcs of $\sigma_s$ decompose the surface into finitely many tiles homeomorphic to $\poly{n_1},\ldots, \poly{n_p}$, and the trunk homeomorphic to $\mob[d+1]$, containing the vertices $i_1,\ldots, i_{d+1}$ of $\mob[n+1]$.
	 Then we have that
	\[ \Link \sap {Y^{d+1}}\simeq \ac{ \poly{n_1} } \Join \ldots \Join  \ac{ \poly{n_p} } \Join  \ack {C}{\mob[ d]}.\]
	This is because any b-arc inside the trunk $\mob[d+1]$ gives a sapling of degree less than $d$. By induction hypothesis, $Y^{d}$ does not contain such a sapling. 
	From Theorem \ref{inner mobius}, we have that $\ack {C}{\mob[ d]}$ is collapsible. So from Proposition \ref{Welker}\ref{join}, the link of the sapling $\sap$ is collapsible. Finally, from Lemma \ref{Welker}\ref{link} we get that $Y^{d-1}$ collapses onto $\fdel{\sigma_s}{Y^{d-1}}$.
	%its tileset is of the form \[  \tile[\sap]= \{\poly{n_1},\ldots, \poly{n_p}, \holed{d_0+1}\} .\] 
	%Now let $\sigma$ be a boundary simplex in $Y^{d}$ containing $\sap$. Then $\Stem =\sap$ -- otherwise we have $\deg\Stem<d$, which is not possible because by the induction hypothesis, $Y^{d'}$ does not contain any boundary simplex that contain a sapling of degree $d'$.
Now if possible, let $\eta$ be a simplex of $Y^d$ containing 2 saplings of degree $d+1$. Then the degree of $\eta$ is necessarily less than $d+1$. By induction hypothesis, $Y^d$ cannot contain such a simplex. See Fig. \ref{disjsap}. Thus we can face delete the saplings of degree $d+1$ in any order. Hence we have that $Y^{d}$ collapses onto $Y^{d+1}$. This finishes the induction step. 
	Since $Y^{n-1}$ does not contain any boundary simplex, it is isomorphic to $\ack C{\holed n}$, which is collapsible. 
	%Finally, the saplings of degree $n-1$ are the minimal b-arcs $m_i$, for $i\in \intbra$. These are the only b-arcs in $Y^{n-2}$. The link of these arcs are isomorphic to $\Del^{n-2}$, which is collapsible. Hence $Y^{n-2}$ collapses onto $\fdel{[m_i]}{Y^{n-2}}$. A simplex containing two or more such saplings has degree less than $n-1$. So it does not belong in the complex $Y^{n-2}$. So we can collapse the minimal b-arcs in any order. Thus $Y^{n-2}$ collapses onto 
\end{proof}

\begin{cor}\label{acmobball}
For $n\geq 1$, the full arc complex of a non-orientable crown is a combinatorial ball of dimension $n-1$.
\end{cor}
\begin{proof}
The proof is similar to that of Corollary \ref{acrownball}.
\end{proof}

\subsection{The full arc complex of a non-orientable crown}

%\begin{lem}
%The arc complex $\ac{\mob [4]}$ is not strongly collapsible. 
%\end{lem}

\begin{prop} \label{acrossnotacone}
For $n\geq 4$, the inner arc complex $\across$ is not a cone. 
\end{prop}

\begin{figure}[th]
\centering
\includegraphics[width=10cm]{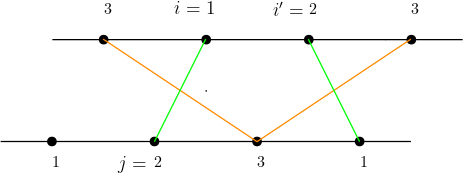}
\caption{ Proposition \ref{acrossnotacone}: $j<i'\leq n$}
\label{acrossnotcone}
\end{figure}
\begin{proof}
	We need to show that for $n\geq 4$, there does not exist any c-arc that is disjoint from every other c-arc. We prove this by contradiction. If possible, let $(i,j)$ be such an arc, with $1\leq i\leq j\leq n$. WLOG, we can assume that $i=1$. Furthermore, since any two maximal c-arcs intersect, we can assume that $1<j$. Since $n\geq 3$, there exists a vertex $i'$ such that either $1<i'<j$ or $j<i'\leq n$. See Fig. \ref{acrossnotcone}. Then the arc $L_{i'}$ intersects $(1,j)$. Hence there cannot be any c-arc that is disjoint from every other c-arc.  
\end{proof}

\begin{prop}\label{bnotdom}
For $n\geq 2$, there does not exist any b-arc that is disjoint from every c-arc of $\mob$. 
\end{prop}
\begin{proof}
Any b-arc separates a polygon $\poly k$ from the surface, where $3\leq k\leq n$. In particular, there is a vertex $i$ on $\poly k$ that is different from the endpoints of the given $b$-arc. Any c-arc with an endpoint on $i$ intersects this b-arc. This concludes the proof. 
\end{proof}
\begin{prop}\label{notacone}
For $n\geq 3$ and $k\in \intbra$, the complex $\ac\mob \smallsetminus \{ [M_1], \ldots, [M_k]\}$ is not a cone.
\end{prop}
\begin{proof}
This follows from the proofs of Propositions \ref{acrossnotacone}, \ref{bnotdom}.
\end{proof}

\begin{lem}\label{firstcoll} 
For every $n\geq 1$, the only vertex-dominated 0-simplices in $\ac\mob$ are given by $[M_i]$, for $i\in \intbra$.
\end{lem}
\begin{proof}
In the base step of the proof of Theorem \ref{mobboundarycollapse}, we showed that for every $i\in\intbra$, the 0-simplex $[M_i]$ is vertex-dominated by $[L_i]$.  Next, the link in $\ac\mob$ of any c-arc is a sphere, so it cannot be a cone. So no c-arc can be vertex-dominated. 
Finally consider any non-maximal b-arc $a$. It decomposes the surface into two tiles homeomorphic to $\mob[m]$ and $\poly k$, with $m\in \intbra[2,n]$ and $k\in \intbra[3,n]$. Therefore, $\Link{[a]}{\ac\mob}= \ac{\mob[m]}\Join \ac{\poly k}$. Since $a$ is not a maximal arc, we have that $m>1$ and so neither of the smaller arc complexes are a 0-simplex. Hence, $\Link{[a]}{\ac\mob}$ is not a cone. 
\end{proof}
\begin{figure}[th]
	\centering
	\includegraphics[width=10cm]{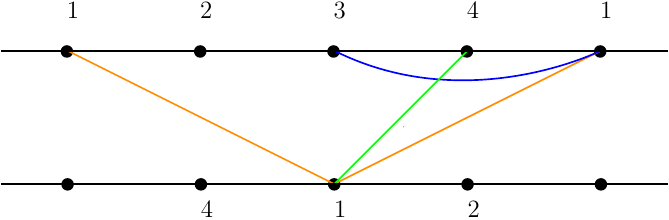}
	\caption{ Proposition \ref{secondcollapse}: $[L_1]$ is not vertex-dominated in $X_1$.}
	\label{longnotvdom}
\end{figure}

\begin{lem}\label{secondcollapse}
Let $n\geq 2$ and $j_0\in \intbra$. Then the only vertex-dominated 0-simplices of $ X_1:=\ac\mob \smallsetminus\{[M_{j_0}]\}$ are $\{[M_j] \mid j\in \intbra\setminus \{j_0\}\}$.
\end{lem}
\begin{proof} WLOG, we can assume that $j_0=1$.
The links of all the arcs that intersect $M_1$ remain unchanged after the removal of $[M_1]$ from $\ac\mob$. In particular, the 0-simplices $[M_2], \ldots, [M_n]$ remain vertex-dominated, while the 0-simplices corresponding to every other c-arc intersecting $M_1$, remain non-dominated. Now we consider the arcs that are disjoint from $M_1$:
\begin{enumerate}
\item the maximal c-arc $L_1$,
\item the b-arcs in the polygonal tile of $M_1$.
\end{enumerate}
Firstly, we show that $[L_1]$ cannot be vertex-dominated. We prove this by contradiction. See Fig. \ref{longnotvdom}. Suppose that there exists an arc $\al$ which is disjoint from any arc that is disjoint from $L_1$. If the arc $\al$ is a c-arc then it must be of the form $\al=(1, j)$ where $j\in \intbra[2,n].$ But the b-arcs $a_{j-1}^{j+1}, a_{j-1}^1$ are both disjoint to $L_1$ and at least one of them intersects $a$.  In Fig. \ref{longnotvdom}, we have $j=4=n$; the b-arc $a_3^1$ intersects the c-arc $(1,4)$. So we have a contradiction. Next we suppose that $\al$ is a b-arc. Then it separates a vertex, say $j$ from the rest of the surface, where $j\in \intbra[2,n]$. Then the c-arc $(1,j)$ is disjoint from $L_1$ and intersects $\al$. So the 0-simplex $[L_1]$ cannot be vertex-dominated in $X_1$. 

Finally,  we check for vertex-domination in the case of b-arcs that are disjoint from $M_1$. The link of such an arc $\be$ in $X_1$ becomes
 $$\Link{[\be]}{X_1}= \left(\ac{\mob[m]}\smallsetminus [M_1]\right)\Join \ac{\poly k} . $$ If $m\geq 3$, then from Proposition \ref{notacone}, we know that the complex $\ac{\mob[m]}\smallsetminus [M_1] $ is not a cone. So the 0-simplex $[b]$ is not vertex-dominated. Since $\be\neq M_1$, we have $m>1$. So we need to treat the remaining case $m=2$.  From Fig. \ref{acmob}, we conclude that the complex $\ac{\mob[m]}\smallsetminus [M_1] $ is not a cone.
 \end{proof}
%Since all the $0$-simplices $[M_2], \ldots, [M_n]$ were vertex-dominated in $\ac\mob$, they remain so in $X_1$ as well. Next we only need to check vertex-domination for 0-simplices  that were in the link of $[M_1]$ in $\ac\mob$. These are $[L_1]$ and the 0-simplices corresponding to all the b-arcs disjoint to $M_1$. 
%Finally, we check vertex domination for 0-simplices corresponding to some b-arc $b$ disjoint to $M_1$. Such an arc is of the form $b=a_{ij}$ with $1\leq i<j\leq n$ or of the form $a_{ji}$ with $i=1,\, j\in \intbra[2,n-1]$. In the first case, the arc decomposes the surface into a polygon $\poly {j-i+1}$ and a Mobius strip $\mob[n-j+i+1]$. The link of $b$ is then given by
%\[ \Link{[b]}{X_1}= \ac{\poly {j-i+1}} \Join (\ac\mob \smallsetminus {[M_1]}) .\] 
%From Lemma \ref{bnotdom}, we know that $[b]$ cannot be vertex-dominated by any boundary arc. When $n-j+i+1\geq 3$, we can apply Lemma \ref{acrossnotacone} to conclude that the arc $b$ cannot be dominated by any c-arc. So we only need to check the case of $b=a_{1n}$. There only 3 c-arcs disjoint from $a_{1n}$, namely, $L_1, L_n$ and $(1,n)$. Since the first two intersect each other, the only possibility for vertex-domination is the third one. But the maximal b-arc $M_n$ intersects the c-arc $(1,n)$ while being disjoint from $a_{1n}$. Hence,  the complex $\ac\mob\smallsetminus {[M_1]} $ cannot be a cone. This implies that $[b]$ isn't vertex-dominated. The second case $a_{ji}$ is treated similarly. 

\begin{lem}\label{vdomaij}
Let $n\geq 4$ and $\mathcal J_n:=\{(i, i+1) \mid  i\in \intbra[1,n-1] \} \cup \{(n,1)\}$. Let $I \subset \intbra[1,n]$ of size at least 2. Then the 0-simplices that are vertex-dominated in   $\ac{\mob} \smallsetminus\{[M_k]\mid k\in I\}$ are given by $\{[M_{j}]\}$ for $j\in \intbra\setminus I$ and $\{a_{j}^i\}$ for $(i,j)\in \mathcal J_n\cap I\times I$. 
\end{lem}
\begin{proof}
Let $X_I:=\ac{\mob} \smallsetminus\{[M_k]\mid k\in I\ $.
Since any two maximal b-arcs always intersect, we have that for every $j\in \intbra\setminus I$, $\Link{[M_j]}{X_1}=\Link {[M_j]}{\ac \mob}$, which is a cone. So the 0-simplices $[M_j]$, $j\in \intbra\setminus I$, remain vertex-dominated. 
 It suffices to check for vertex-dominated arcs that are disjoint from the maximal b-arcs that were removed:
\begin{enumerate}
\item the b-arcs in the polygonal tile of $M_j$.
\item the maximal c-arcs $L_j$ for $j\in I$,
\end{enumerate}
%The argument for the arcs that are disjoint from $M_2$ but intersects $M_1$ is the same as above. We only have to check the vertex-domination for the 0-simplices that are in the intersection $\Link{[M_1]}{\ac\mob} \cap \Link{[M_2]}{\ac\mob}$. These correspond to the arcs that are disjoint from both $M_1, M_2: a_{21}$ and all the b-arcs that lie in the polygonal tile of $a_{21}$. 
\begin{figure}[th]
	\centering
	\includegraphics[width=12cm]{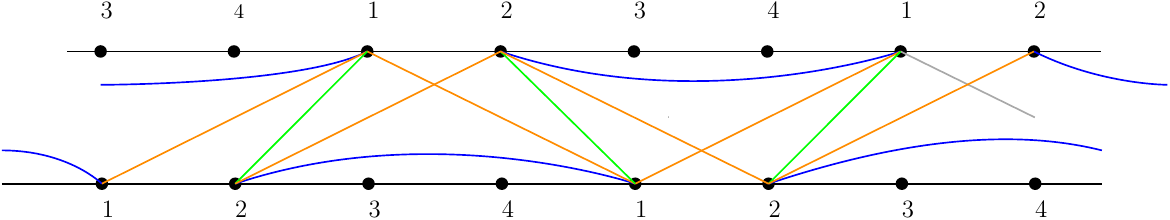}
	\caption{Lemma \ref{secondcollapse}: $[a_{21}]$ is vertex-dominated in $[(1,2)]$ in $X_I$.}
	\label{linka21}
\end{figure}

Firstly, we consider a 0-simplex $[a_{j}^i]$ with $(i,j)\in \mathcal J_n\cap I\times I$. We claim that it is vertex dominated by the 0-simplex $[(i,j)]$. In Fig \ref{linka21}, we have $n=4, i=1, j=2$. The arc $a_j^i$ decomposes the surface into a polygon $\poly p$ ($p\geq 3$) and a non-orientable crown $\mob[2]$, containing the two vertices $i,j$. So the link of this 0-simplex in $X_I$ is given by 
\begin{align*}
\Link{[a_j^i]}{X_I}&= \ac{\poly p}\Join \left(\ac{\mob [2]}\smallsetminus \{[M_i], [M_j]\}\right)\\
&= \s{p-4} \Join \s 0 \Join [(1,2)],
\end{align*}
where the 0-sphere corresponds to the two intersecting maximal c-arcs $L_i, L_j$.
%Hence $[a_{j}^i]$ is vertex-dominated in $X_I$, when $(i,j)\in \mathcal J_n\cap I\times I$. \\ Next let us check for vertex-domination for arcs $a_{j}^i$ with $(i,j)\in \mathcal J_n\setminus I\times I$. 
%We have that 
%\[
%\Link{[a_j^i]}{X_I}= \left\{ 
%\begin{array}{lc}
%\ac{\poly p}\Join \left(\ac{\mob [2]}\smallsetminus \{[M_i]\}\right), & \text{ when } i\in I,\\
%\ac{\poly p}\Join \left(\ac{\mob [2]}\smallsetminus \{[M_j]\}\right), & \text{ when } j\in I,\\
%\ac{\poly p}\Join \ac{\mob [2]}& \text{ when } (i,j)\in \mathcal J_n\setminus I\times I,
%\end{array}
%\right.
%\]
%Since none of the cases is a cone, we get that these b-arcs are not vertex-dominated in $X_I$. 
Next, consider any b-arc in the polygonal tile of  $M_j$, with $j\in I$ such that neither $(j-1,j)$ nor $(j,j+1)$ lie in $ \mathcal J_n\cap I\times I$. Then the argument is same as that in the the second case of Lemma \ref{secondcollapse}. \\
Finally, consider the maximal c-arc $L_j$ for $j\in I$. Then its link is given by $\Link{[L_j]}{\ac\mob \smallsetminus [M_j]}$ which is not a cone from the proof of Lemma \ref{secondcollapse}.
%\[
%\Link{[L_j]}{X_I}= \left\{ 
%	\begin{array}{lc}
%\Link{[L_j]}{\ac\mob \smallsetminus [M_j]}, & \text{ when } j\in I\\
%\Link{[L_j]}{  \ac \mob  }, & \text{ when } j\in \intbra\setminus I.
%	\end{array}
%	\right.
%\]

%Next, we check vertex domination for b-arcs that are different from $a_{j}^i$ with $(i,j)\in \mathcal J_n\cap I\times I$.
%Now any b-arc $b$ that lies in the polygonal tile of $a_{j}^i]$ for some $(i,j\in I)$, divides the surface into a polygon $\poly q$ and a non-orientable crown with at least 3 vertices. So its link is given by
%$$\Link{[b]}{X_2}= \ac{\poly k} \Join  (\ac{\mob [l]}\smallsetminus \{[M_j]\mid  j\in I\} ) ,$$ where $l\geq 3$. Using Proposition \ref{notacone}, we get that this link is not a cone. 

\end{proof}

\begin{lem}\label{lemdecrease}
	Let $n, I, \mathcal J_n$ be as in Lemma \ref{vdomaij}. Let $J\subset \mathcal J_n\cap I\times I$. Then the 0-simplices that are vertex-dominated in $\ac{\mob} \smallsetminus\{[M_j]\mid j\in I \} \cup \{[a_j^i] \mid (i,j)\in J \}$ are given by $\{[M_{j}]\}$ for $j\in \intbra\setminus I$ and $\{a_{j}^i\}$ for $(i,j)\in \mathcal J_n\cap I\times I\setminus J$. 
\end{lem}
\begin{proof}
	Denote $X_{I,J}:= \ac{\mob} \smallsetminus\{ [M_j]\mid j\in I \} \cup \{[a_j^i] \mid (i,j)\in J \}$.
It suffices to check for vertex-dominated 0-simplices in the links of the 0-simplices that were removed: $\{ [M_j]\mid j\in I \} \cup \{[a_j^i] \mid (i,j)\in J \}$. These correspond to 
\begin{enumerate}
\item the maximal c-arcs $L_j$ for $j\in I$,
\item the non-maximal c-arcs $ (i,j)$, for $i,j \in J$,
\item all the b-arcs in the polygonal tile of $a_j^i$, for $i,j \in J$,
\item all the b-arcs in the polygonal tile of $M_j$, for $j\in I$ such that neither $(j-1,j)$ nor $(j,j+1)$ lie in $ \mathcal J_n\cap I\times I$.
\end{enumerate}
Firstly, we show that the maximal c-arcs $L_j$ for $j\in I$ cannot be vertex-dominated.
Cutting the surface along $L_j$, we get a polygon $\poly {n+2}$.
The link of $[L_j]$ is given by \[
\Link{[L_j]}{X_{I,J}}= \ac{\poly {n+2}} \smallsetminus \left\{ 
\begin{array}{lc}
\{[M_j]\}, & \text{ when } j\in I \text{ and } (j-1, j), (j,j+1)\notin J,\\
\{[M_j], a_j^{j-1}\}, & \text{ when } (j-1,j)\in J \text{ and } (j,j+1)\notin J,\\
\{[M_j], a_{j+1}^{j}\}, & \text{ when } (j,j+1)\in J \text{ and } (j-1,j)\notin J,\\
\{[M_j], a_j^{j-1}, a_{j+1}^{j}\}, & \text{ when } (j-1, j), (j,j+1)\in J\\
\varnothing & \text{ otherwise.} 
\end{array}
\right.
\] Suppose that a c-arc vertex-dominates $L_j$. Then it is of the form $(j,i)$, for $i\in \intbra[1,j-1]$ or $(i,j)$ for $i\in \intbra[j+1,n]$. See Fig. \ref{decrease}. For $j<i< n$ or $i<n=j$, the b-arc $a_{i-1}^{i+1}$ intersects the arc $(j,i)$ or $(i,j)$ respectively. When $j<i=n$, the b-arc $a_{i+1}^{i-1} $ intersects $(j,i)$. Since $n\geq 4$, these b-arcs are different from the three arcs $M_j, a_{j}^{j-1}, a_{j+1}^{j}$. So we get that $L_j$ cannot be vertex dominated by a c-arc. We have already seen in the proof of Lemma \ref{secondcollapse}, that a b-arc cannot dominate a c-arc. So c-arcs $L_j$, for $j\in I$, cannot be vertex-dominated in $X_{I,J}$. 

\begin{figure}[th]
	\centering
	\includegraphics[width=14cm]{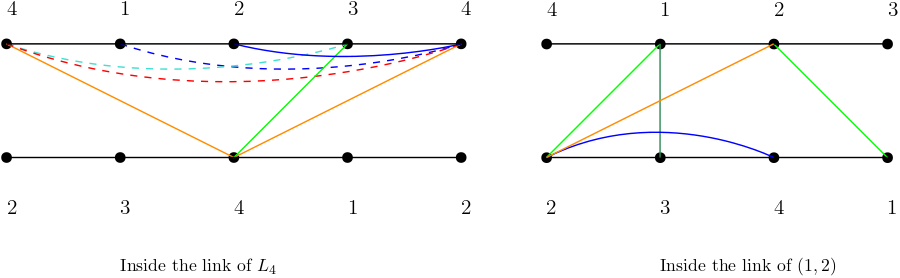}
	\caption{Lemma \ref{lemdecrease}}
	\label{decrease}
\end{figure}
Next we prove that the non-maximal c-arcs $ (i,j)$, for $i,j \in J$, cannot be vertex-dominated.
Suppose that $(i,j)$ is vertex-dominated by another c-arc. See  the right panel of Fig. \ref{decrease}. Either $j=i+1$ or $j=n, i=1$. Suppose that the b-arc  is vertex-dominated by c-arc $\al$. Then at least one endpoint of $\al$ lies on either $i$ or $j$. WLOG, suppose that $\al$ has an endpoint on $i$. Then the maximal arc $L_j$ intersects $\al$. So $\al$ cannot vertex dominate $(i,j)$. We have already seen that a b-arc cannot vertex-dominate. Thus the $(i,j)$ cannot be vertex-dominated in $X_{I,J}$. 

Next suppose that $\al$ is a b-arc in the polygonal tile of $a_j^i$, for $i,j\in J$. It divides the surface into a polygon $\poly p$ and a non-orientable crown $\mob[m]$, with $m\geq 3$. Then its link is given by 
\[
\Link{[\al]}{X_{I,J}}= \ac{\poly p}\Join \left(\ac{\mob[m]}\smallsetminus \{ [M_k] \mid \al \in \Link{[M_k]}{\ac\mob} \}\cup \{[a_j^i] \mid \al\in \Link{[a_j^i]}{\ac\mob}   \}\right)
\]
From Propositions \ref{notacone}, \ref{bnotdom}, we get that neither a b-arc nor a c-arc can vertex-dominate $\al.$\\
Finally, the argument for the 4th type of arcs is identical to that in Lemma \ref{secondcollapse}.

This concludes the proof. 
\end{proof}

\begin{thm}
	For $n\geq 4$, the full arc complex $\ac{\mob}$ of the surface $\mob$ is not strongly collapsible. 
\end{thm}
\begin{proof}
From Lemma \ref{firstcoll}, we get that the only way to strongly collapse the full arc complex is to remove the 0-simplex corresponding to some maximal b-arc $M_j$, $j\in \intbra$. From Lemma \ref{secondcollapse}, we get that the only way to strongly collapse the complex $\ac\mob \setminus \{M_j\}$ is to remove another maximal b-arc.. Next from Lemma \ref{vdomaij}, we get that if we remove two maximal b-arcs $M_i, M_j$ such that $j=i+1$ are consecutive or $i=n, j=1$, then the b-arc $a_j^i$ becomes vertex-dominated. From Lemma \ref {lemdecrease}, we get that we can remove the 0-simplices in $\mathcal D:=\{[M_j] \mid j \in \intbra \}\cup\{[a_j^i] \mid (i,j)\in \mathcal J_n\}$ in any order and we will always end up with the complex $\ac\mob \smallsetminus \mathcal D$. Finally, from Lemma \ref{lemdecrease}, we get that this complex is minimal. Thus the full arc complex is not strongly collapsible.
\end{proof}

\subsection{Integral strip}
In this section we show that the arc complex of an integral strip is strongly collapsible.
\begin{thm}\label{thmstrip}
	For $m,n\geq1$ and $m+n\geq 5$, the arc complex $\acs {m,n}$ is strongly collapsible. 
\end{thm}
\begin{proof}
	For all $m\geq 1$ and $n\geq 1$, the arc complexes $\acs{m,1}$ and $\acs{1, n}$ are simplices of dimension $m-1$ and $n-1$, respectively. So they are strongly collapsible. See Fig. \ref{intstripbase}. Next we verify the statement for $m=3,n=2$. There are four arcs $(1,3), (1,2), (2,2), (2,1)$. The 0-simplices $(1,3)$ and $(2,1)$ are vertex-dominated by the 0-simplices $(1,2)$ and $(2,2)$, respectively. Removing the two former arcs we get a 1-simplex which is collapsible. This concludes the base step.
	
	Suppose that for $1\leq m'\leq m$, $1\leq n' \leq n$ and $m'+n'\geq 5$, the arc complexes $ \acs {m'+1, n'}  $ and $\acs {m', n'+1}$ are strongly collapsible. For the induction step, we need to show that the arc complex $\acs {m'+1, n'+1}$ is strongly collapsible. The idea is to strongly collapse all the new arcs in $P(m+1,n+1)$ have an endpoint on the red vertex $n+1$. The only other new arc is $(m+1, n)$ which forms a boundary edge for the strip $P(m+1,n)$. 
	
	For the ease of notation, let us denote $a_i:=[(i, n+1)]$, for all $i\in \intbra [1,m]$. Then we show that 
	\begin{align*}
	\acs{  m+1,n+1 }  &\Searrow \acs{m+1,n+1}\smallsetminus\{a_1\} \\
	& \Searrow \ldots \\
	&\Searrow \acs{m+1,n+1}\smallsetminus\{a_1, \ldots, a_m\}
	%%&\Searrow \acs{m+1, n}
	\end{align*}
	The arc $a_1$ decomposes the integral strip $P(m+1, n+1)$ into two smaller strips $P(1,n+1)$ and $P(m+1, 1)$. The link of the 0-simplex $[a_1]$ is given by $$\Link{a_1}{ \acs{m+1,n+1}}= \fan{{\color{blue}1}}{1, n}\Join \fan{\color{red}n+1}{1, m} , $$ which is a cone. So we get that $\acs{ m+1,n+1 }  \Searrow \acs{m+1,n+1} \smallsetminus \{a_1\}$.
	
	\begin{figure}[th!]
		\centering
		\includegraphics[width=10cm]{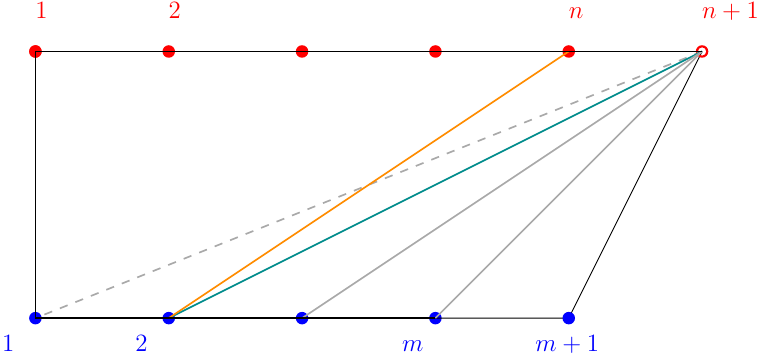}
		\caption{In Theorem \ref{thmstrip}, The 0-simplex $a_2=[(2,n+1)]$ is vertex-dominated by $[(2,n)]$.}
		\label{figstrip}
	\end{figure}
	
	Next suppose that for $2 \leq i'\leq i\leq m$,  $\ac{P(m+1,n+1)}\smallsetminus\{ a_1,\ldots, a_{i'-1} \} \Searrow \ac{P(m+1,n+1)}\smallsetminus\{ a_1,\ldots, a_{i'} \} $. 
	We need to show that $a_{i+1}$ is vertex-dominated in $\ac{P(m+1,n+1)}\smallsetminus\{ a_1,\ldots, a_i \}$.
%	$$\fdel{ a_{i-1}  }{\ldots, \fdel{a_1}{ \ac{P(m_0+1,n_0+1)} }} \searrow \fdel{ a_{i}}{ \fdel{a_{i-1}   }  {\ldots, \fdel{a_1}{ \ac{P(m_0+1,n_0+1)} }}}.$$\\
	See Fig. \ref{figstrip}. The arc $a_{i+1}$ decomposes the integral strip $P(m+1, n+1)$ into two smaller strips $P( i+1, n+1)$ and $P(m+2-i, 1)$. Since then arcs $(i, n+1)$ have been removed, any triangulation of $P(i+1, n+1)$ must contain the arc $(i+1,n)$. So we get that the 0-simplex $a_{i+1}$ is vertex-dominated by $[(i+1,n)]$. This concludes the induction on $i. $
	Finally, we have that $\acs{m+1,n+1}\smallsetminus\{a_1, \ldots, a_m\}= \acs{m+1,n} \Join [(m+1, n)]$, which is strongly collapsible. This concludes the proof.
	%So the link of the 0-simplex $a_{i+1}$ is given by $$\Link{a_{i+1}}{ \acs {m+1,n+1} \smallsetminus \{ a_1,\ldots, a_i \}  }= \ac {P ( m_0+1,i_0+1 )}   \Join \fan{y_{m_0+1}}{1, x_{n_0+1}} , $$ which is collapsible by induction hypothesis. This concludes the induction step for $i$.
	
	%Finally, the link of the 0-simplex $a_0:=[( x_{n_0+1},y_{m_0} )]$ is isomorphic to the arc complex of $P(m_0,n_0+1)$, which is collapsible by the induction hypothesis. So, from Lemma \ref{link}, $X\searrow \fdel{  a_0}{  X}= \ac{  P(m_0,n_0+1)  } $. Thus we have, 	$\ac{  P(m_0+1,n_0+1)  }  \searrow  \ac{  P(m_0,n_0+1)  }$, which finishes the induction step for $m,n$.
\end{proof}
	\newpage
	\bibliographystyle{plain}
	\bibliography{acholedmobius.bib}

\begin{thebibliography}{10}

\bibitem{catzero}
Karim Adiprasito and Bruno Benedetti.
\newblock Collapsibility of cat(0) spaces.
\newblock {\em Geometriae Dedicata}, 206(1):181--199, 2020.

\bibitem{extremal}
Karim~A. Adiprasito, Bruno Benedetti, and Frank~H. Lutz.
\newblock Extremal examples of collapsible complexes and random discrete morse
  theory.
\newblock {\em Discrete \& Computational Geometry}, 57(4):824--853, 2017.

\bibitem{strong}
Jonathan~Ariel Barmak and Elias~Gabriel Minian.
\newblock Strong homotopy types, nerves and collapses.
\newblock {\em Discrete \& Computational Geometry}, 47(2):301--328, 2012.

\bibitem{dmtbdry}
Bruno Benedetti.
\newblock Discrete morse theory for manifolds with boundary.
\newblock {\em Transactions of the American Mathematical Society},
  364(12):6631--6670, 2012.

\bibitem{benball}
Bruno Benedetti and Frank~H. Lutz.
\newblock The dunce hat in a minimal non-extendably collapsible 3-ball.
\newblock {\em arXiv: Algebraic Topology}, 2009.

\bibitem{bing}
R.H. Bing.
\newblock Some aspects of the topology of 3-manifolds related to the poincaré
  conjecture.
\newblock {\em Lectures on Modern Mathematics}, 2(1):93--128, 1964.

\bibitem{PHflag}
Jean-Daniel Boissonnat and Siddharth Pritam.
\newblock Computing persistent homology of flag complexes via strong collapses.
\newblock In {\em 35th International Symposium on Computational Geometry (SoCG
  2019)}. Schloss Dagstuhl-Leibniz-Zentrum fuer Informatik, 2019.

\bibitem{PHstrcoll}
Jean-Daniel Boissonnat, Siddharth Pritam, and Divyansh Pareek.
\newblock Strong collapse and persistent homology.
\newblock {\em Journal of Topology and Analysis}, 15(01):185--213, 2023.

\bibitem{crowley}
Katherine Crowley.
\newblock Simplicial collapsibility, discrete morse theory, and the geometry of
  nonpositively curved simplicial complexes.
\newblock {\em Geometriae Dedicata}, 133(1):35--50, 2008.

\bibitem{danaraj}
Gopal Danaraj and Victor Klee.
\newblock {Shellings of spheres and polytopes}.
\newblock {\em Duke Mathematical Journal}, 41(2):443 -- 451, 1974.

\bibitem{Harvey}
W.~J. Harvey.
\newblock {\em Boundary Structure of The Modular Group}, pages 245--252.
\newblock Princeton University Press, Princeton, 1981.

\bibitem{worst}
Davide Lofano and Andrew Newman.
\newblock The worst way to collapse a simplex.
\newblock {\em Israel Journal of Mathematics}, 244(2):625--647, 2021.

\bibitem{LC}
Jiří Matoušek.
\newblock Lc reductions yield isomorphic simplicial complexes.
\newblock {\em Contributions to Discrete Mathematics [electronic only]}, 3, 01
  2008.

\bibitem{ppballs}
Pallavi Panda.
\newblock The arc complexes of bicoloured polygons are balls.
\newblock 2023.
\newblock arXiv: math/2306.06695.

\bibitem{pointi}
Hugo Parlier and Lionel Pournin.
\newblock Once punctured disks, non-convex polygons, and pointihedra.
\newblock {\em Annals of Combinatorics}, 22(3):619--640, 2018.

\bibitem{sphereconj}
R.~C. Penner.
\newblock {The structure and singularities of quotient arc complexes}.
\newblock {\em Journal of Topology}, 1(3):527--550, 05 2008.

\bibitem{penner}
R.~C. Penner.
\newblock {\em Decorated {Teichmüller} {Theory}}.
\newblock The {QGM} master class series. European Mathematical Society,
  Zürich, Switzerland, 2012.
\newblock OCLC: ocn773019621.

\bibitem{diameterpournin}
Lionel Pournin.
\newblock The diameter of associahedra.
\newblock {\em Advances in Mathematics}, 259:13--42, 2014.

\bibitem{tarjan}
Daniel~D. Sleator, Robert~E. Tarjan, and William~P. Thurston.
\newblock Rotation distance, triangulations, and hyperbolic geometry.
\newblock {\em J. Amer. Math. Soc.}, 1:647--681, 1988.

\bibitem{welker}
Volkmar Welker.
\newblock Constructions preserving evasiveness and collapsibility.
\newblock {\em Discrete Mathematics}, 207(1):243--255, 1999.

\bibitem{whitehead}
J.~H.~C. Whitehead.
\newblock Simplicial spaces, nuclei and m-groups.
\newblock {\em Proceedings of the London Mathematical Society},
  s2-45(1):243--327, 1939.

\bibitem{wilson}
Jon Wilson.
\newblock Shellability and sphericity of finite quasi-arc complexes.
\newblock {\em Discrete \& Computational Geometry}, 59(3):680--706, 2018.

\bibitem{zeeman}
E.C. Zeeman.
\newblock On the dunce hat.
\newblock {\em Topology}, 2(4):341--358, 1963.

\end{thebibliography}
	
\end{document}